\documentclass[11pt]{article}
\usepackage{amsmath,amsfonts,amssymb,amsthm,epsfig,epstopdf,url,array} 
\usepackage{amsmath}
\usepackage[english]{babel}
\usepackage{amsfonts}
\usepackage{hyperref}
\usepackage{color}

\numberwithin{equation}{section}

\newtheorem{theo}{Theorem}[section]
\newtheorem{lem}[theo]{Lemma}
\newtheorem{cor}[theo]{Corollary}
\newtheorem{rem}{Remark}[section]
\newtheorem{prop}[theo]{Proposition}

\newcommand{\mysection}[1]{\section{#1} \setcounter{equation}{0}}

\newcommand{\be}{\begin{equation} \label}
\newcommand{\ee}{\end{equation}}
\newcommand{\bea}{\begin{eqnarray}\label}
\newcommand{\eea}{\end{eqnarray}}
\newcommand{\bas}{\begin{eqnarray*}}
\newcommand{\eas}{\end{eqnarray*}}
\newcommand{\bit}{\begin{itemize}}
\newcommand{\eit}{\end{itemize}}

\newcommand{\R}{\mathbb{R}}
\newcommand{\N}{\mathbb{N}}

\newcommand{\eps}{\varepsilon}

\newcommand{\ts}{\textstyle}
\newcommand{\ds}{\displaystyle}

\usepackage{titlesec}

\usepackage{lipsum}

\makeatletter

\usepackage[twoside]{fancyhdr}
\newcommand*{\runinsubsection}{%
  \@startsection{subsection}%
  {2}
  {\z@}
  {-3.25ex\@plus -1ex \@minus -.2ex}
  {-1.5em \@plus -.1em}
  {\normalfont\large\bfseries}
}

\makeatother

\begin{document}
\title
{Sharp macroscopic blow-up
behavior for the parabolic-elliptic Keller-Segel system in dimensions $n\ge 3$}
\author{Loth Damagui CHABI and Philippe SOUPLET\footnote{chabi@math.univ-paris13.fr,  souplet@math.univ-paris13.fr}\\
{\small Universit\'{e} Sorbonne Paris Nord, CNRS UMR 7539}\\
{\small Laboratoire Analyse, G\'{e}om\'{e}trie et Applications, 93430, Villetaneuse, France} 
}
\date{}
\maketitle

\begin{abstract} \noindent
We study the space-time concentration or 
blow-up asymptotics of radially decreasing solutions 
of the parabolic-elliptic Keller-Segel system in the whole space or in a ball.
We show that, for any solution in dimensions $3\le n\le 9$ (assuming finite mass in the whole space case),
there exists a nonflat backward self-similar solution $U$ such that
$$u(x,t)=(1+o(1))U(x,t),\quad\hbox{as $(x,t)\to (0,T)$.}$$
This macroscopic behavior is important from the physical point of view, 
since it gives a sharp description of the concentration phenomenon in the scale of the original space-time variables~$(x,t)$. 
It strongly improves on existing results, since such behavior was previously known (\cite{GMS}) to hold only  
in the microscopic scale $|x|\le O(\sqrt{T-t})$ 
as $t\to T$ (and in the whole space case only). 
As a consequence, we obtain the two-sided global estimate
$$C_1\le (T-t+|x|^2)u(x,t)\le C_2\quad\hbox{in $B_R\times(T/2,T)$},$$
whose upper part only was known before (\cite{Soup-Win}),
as well as the sharp final profile:
$$\lim_{x\to 0} |x|^2u(x,T)=L\in(0,\infty).$$
The latter improves, with a different proof, the recent result of \cite{BZ} by excluding the possibility $L=0$.
We also give extensions of these results, in higher dimensions, to type~I and to time monotone solutions.
Moreover, we extend the known results on type I estimates
and on convergence in similarity variables,
and significantly simplify their proofs. 


\end{abstract}

\mysection{Introduction}

Let $\Omega=B_R\subset\R^n$ or $\Omega=\R^n$, with $R>0$ and $n\ge 2$.
In this article we consider radially symmetric solutions of the classical parabolic-elliptic Keller-Segel-Patlak system
\be{0}
        \left\{ \begin{array}{lcll}
    	\hfill u_t &=& \Delta u - \nabla \cdot (u\nabla v), 
    	& x\in\Omega, \ t>0, \\[1mm]
    	\hfill 0 &=& \Delta v + u,
    	& x\in\Omega, \ t>0, \\[1mm]
	\hfill  \frac{\partial u}{\partial\nu}-u \frac{\partial v}{\partial\nu}\ = \ v &=&0,
    	& x\in\partial\Omega, \ t>0, \\[1mm]
    	\hfill u(x,0)&=&u_0(x),
    	& x\in\Omega,
        \end{array} \right.
\ee
where the boundary conditions are omitted in case $\Omega=\R^n$.
As for the initial data, we will assume that
\be{i0}
u_0\in L^\infty(\Omega),\ u_0\ge 0, 
\  \mbox{$u_0$ is radially symmetric and nonincreasing with respect to $|x|$}. 
\ee

This system is a very well-known model of chemotaxis, 
where $u$ and $v$ respectively stand for the density of the bacterial population 
and of the secreted chemoattractant.
System \eqref{0} is also involved in a model of gravitational interaction of particles \cite{BN94, BHN, biler1995}.
It constitutes a very active research topic, 
which has received considerable attention from the mathematical point of view
(see e.g.~the surveys \cite{Horst1,Horst2} and the books \cite{SuzBook,BilerBook} for references).

Problem \eqref{0} is locally well posed (see Section~\ref{secprelim}) 
and we denote by $(u,v)$ its unique, maximal classical solution,
and by $T$ its maximal existence time. 
Moreover, $u$ is nonnegative and radial decreasing under assumption \eqref{i0}.
It is known, see~\cite{jaeger_luckhaus,BHN,biler1995,nagai1995,CPZ,Sen05,BDP}, that if $n\ge 2$ and $u_0$ is suitably large, then $T<\infty$ and the solution blows up in the following sense:
$$	\lim_{t\to T} \|u(\cdot,t)\|_{L^\infty(\Omega)} = \infty.$$ 
Recalling the mass conservation property $\|u(t)\|_1=\|u_0\|_1$,
and keeping in mind the biological background of system \eqref{0},
this can be  interpreted as a phenomenon of concentration or aggregation of the bacterial population.
Understanding the asymptotic behavior of blow-up solutions is thus meaningful 
for the interpretation of the model, especially since
remarkable differences between the cases $n=2$ and $n\ge 3$ have been discovered in previous work 
regarding the form of blow-up singularities 
(see Remark~\ref{remTypeI}  below for more details).

In this paper we are concerned with the case $n\ge 3$ and
one of our main aims is to give a sharp description, {\it in the scale of the original variables~$(x,t)$}, of the space-time and final space 
concentration behavior of solutions of problem \eqref{0} under assumption \eqref{i0}.
Especially, in the range $3\le n\le 9$ (and under a finite mass assumption if $\Omega=\R^n$),
our description will be valid for any radially decreasing blow-up solution.
Our results are in particular motivated by the works \cite{GP, GMS, Soup-Win, BZ} and will sharpen and/or extend the descriptions given there.

\medskip

For a given solution $(u,v)$ of  \eqref{0}, its blow-up set $B(u_0)$ is defined by 
$$
B(u_0):=\bigl\{x_0\in \overline\Omega;\, |u(x_j,t_j)|+|v(x_j,t_j)| \to \infty 
\ \hbox{for some sequence $(x_j,t_j)\to (x_0,T)$}\bigr\}.
$$
Under assumption \eqref{i0}, we have $B(u_0)=\{0\}$ for $\Omega=B_R$.
When $\Omega=\R^n$, this remains true whenever $u_0\in L^1(\R^n)$ (see Remark~\ref{remopen}(i)). 
In what follows, we denote by 
$$u(x,T):=\lim_{t\to T}u(x,t)$$
 the final blow-up profile of $u$. 
If $B(u_0)=\{0\}$ then, by parabolic estimates,
$u(x,T)$ exists and is finite for all $x\in\Omega\setminus\{0\}$.

In view of the statement of our main results, we recall that
for $n\ge 3$, problem \eqref{0} with $\Omega=\R^n$ is known to possess positive, radially symmetric, backward self-similar solutions.
We refer to \cite{HMV, BCKSV, GP, Sen05, GMS, Naito-Senba-2012, NWZ}, where their existence and various properties are obtained.
The $u$-component of such solutions is of the form
$$u(x,t)=\frac{1}{T-t}U\Big(\frac{|x|}{\sqrt{T-t}}\Big),$$
where the similarity profile $U=U(y)>0$ is given by $U=n\Psi+y\Psi_y$, with $\Psi(y)$ any global positive classical solution of the initial value problem
\be{eqnPsi}
\Psi_{yy}+\Big(\frac{n+1}{y}-\frac{y}{2}\Big)\Psi_y-\Psi+\Psi(y \Psi_y+n\Psi)=0,\quad y\in(0,\infty),
\qquad \Psi^\prime(0)=0.
\ee
Denote 
$$\mathcal{S}=\Big\{\hbox{$U=y\Psi_y+n\Psi$, such that $\Psi=\Psi(y)\in C^2([0,\infty))$ is a global positive solution of \eqref{eqnPsi}.}\Big\}$$
Beside the constant profile $U\equiv 1$ (i.e.~$\Psi\equiv 1/n$), corresponding to the flat ODE solution $u(x,t)=(T-t)^{-1}$, the set~$\mathcal{S}$ contains the explicit solution \cite{BCKSV}:
$$U_0(y)= \frac{4(n-2)(2n+|y|^2)}{(2(n-2)+|y|^2)^2}.$$
It is known that $\mathcal{S}$ contains at least a countable family.
Moreover, any nonconstant $U\in\mathcal{S}$ has quadratic decay at infinity, namely
$$\lim_{y\to\infty}y^2 U(y)=C\quad\hbox{for some $C\in(0,\infty)$}.$$
As a consequence, the corresponding self-similar solution
has blowup set $\{0\}$ and final profile $C|x|^{-2}$.
The latter is thus a multiple of the special singular solution $u_c=2(n-2)|x|^{-2}$,
known as the Chandrasekhar solution, which, as an initial data, plays a critical role in the local existence theory (see \cite{BilerBook} and the references therein).

\mysection{Main results}

\subsection{Sharp global space-time blow-up behavior in dimensions $3\le n\le 9$}

Our first main result is a sharp description, {\it in the scale of the original variables~$(x,t)$}, of the space-time concentration behavior
for general radially decreasing blow-up solutions in dimensions $3\le n\le 9$.

\begin{theo}\label{theo13}
Let $3\le n\le 9$ and consider problem \eqref{0} with $\Omega=\R^n$ or $\Omega=B_R$.
Assume that $u_0$ satisfies \eqref{i0},
along with $u_0\in L^1(\R^n)$ in case $\Omega=\R^n$,
and that $T=T(u_0)<\infty$.

\begin{itemize}
\parskip=1pt
 \itemsep=1pt
 
\item[(i)] (Macroscopic self-similar behavior)
There exists $U\in \mathcal{S}\setminus\{1\}$, such that
  \be{13.1}
	 u(x,t) = \frac{1+\eps(x,t)}{T-t}U\Big(\frac{|x|}{\sqrt{T-t}}\Big)
	\qquad \hbox{with } \lim_{(x,t)\to (0,T)} \eps(x,t)=0.
  \ee

\item[(ii)] (Two-sided global space-time estimate)
 There exist $C_1,C_2>0$, such that
  \be{globalasympt}
\frac{C_1}{T-t+|x|^2}\le u(x,t) \le \frac{C_2}{T-t+|x|^2},
	\qquad (x,t)\in B_\rho\times[T/2,T),
\ee
with $\rho=R$ if $R<\infty$ and $\rho=1$ otherwise.

\item[(iii)] (Sharp final blow-up profile) There exists $L>0$ such that 
  \be{finalblow-up}
 u(x,T)=(1+o(1))L|x|^{-2},
	\qquad \hbox{as } x\to 0.
  \ee
  \end{itemize}
\end{theo}

We stress that the macroscopic self-similar behavior \eqref{13.1} is important from the physical point of view,  
as it gives a sharp description of the concentration phenomenon in the scale of the original space-time variables~$(x,t)$.  
It strongly improves on existing results, since such behavior  
was previously known to hold only in the {\it microscopic} scale $|x|\le O(\sqrt{T-t})$ 
as $t\to T$ (cf.~\cite{GMS} 
in the case $\Omega=\R^n$ and \cite{GP} for special classes of radially decreasing solutions in the case $\Omega=B_R$).
Note that this scale does not describe the final profile \eqref{finalblow-up} nor provides global space-time estimates such as \eqref{globalasympt}.
As for the latter, only its upper part was known before (cf.~\cite{Soup-Win}).

Concerning the final profile, under the assumptions of Theorem~\ref{theo13} 
 with $\Omega=\R^n$ and $n\ge 3$,
 it was recently shown in \cite{BZ}, by a quite different argument, that either \eqref{finalblow-up} holds or $\lim_{x\to 0}|x|^2u(x,T)=0$.
Our result rules out the latter possibility in dimensions $3\le n\le 9$, 
thus providing a sharp final profile,
and also covers the case $\Omega=B_R$
(see~Remarks~\ref{remopen}(iii) and \ref{remBZ} for details and further comparison with the result 
and proof in \cite{BZ}).

  \medskip
  
  \begin{rem} {\bf (Stable behavior)}
When $\Omega=\R^3$, it was shown 
 in \cite{GS-2024,Collot-Zhang}, confirming a long standing conjecture from \cite{BCKSV} based on numerical experiments, that the blow-up profile $U_0$ is stable in the following sense:
There exists $\eps>0$ such that, if
$0\le u_0\in L^\infty(\R^3)$ is radially symmetric and $\|u_0-U_0\|_\infty\le\eps$,
then $(u,v)$ blows up at a finite time $T=T(u_0)$ and
\be{conclstabil}
\big\|u(t,x)-(T-t)^{-1}U_0(x/\sqrt{T-t})\big\|_{L^\infty_x(\R^3)}=o((T-t)^{-1}),\quad t\to T.
\ee
By this and Theorem~\ref{theo13}, under the same assumption with $u_0$
radially nonincreasing, 
we obtain the behavior of $u$ in the larger macroscopic scale and its final profile, namely:
\be{conclstabilmacro}
u(x,t)=4(1+\eps(x,t))\frac{x^2+6(T-t)}{(x^2+2(T-t))^2},
	\ \hbox{with}\ \lim_{(0,T)} \eps=0 \quad\hbox{ and } \quad
	u(x,T)\sim 4|x|^{-2},\  x\to 0.
\ee
Very recently \cite{LZ25}, as an important further step, the stability results in \cite{GS-2024,Collot-Zhang} were extended to the nonradial setting.
Also the right hand side of \eqref{conclstabil} was improved to $o((T-t)^{-1+\delta})$ for some (nonexplicit) small $\delta>0$.
We stress again that, since $\lim_{|y|\to\infty} U_0(y)=0$, estimate \eqref{conclstabil} with such right hand sides
does not provide the macroscopic behavior nor the final profile like \eqref{conclstabilmacro}.

\end{rem}

\subsection{Extensions to type I and to monotone in time solutions} 

Recall that any blow-up solution of problem \eqref{0} satisfies $\liminf_{t\to T} \, (T-t)\|u(t)\|_\infty>0$ and that
blow-up is said to be of type I if 
$$\|u(t)\|_\infty\le M(T-t)^{-1},\quad 0<t<T,$$
for some constant $M>0$, and type~II otherwise. This classification is motivated by scale invariance considerations
 and the underlying ODE. Indeed, substituting the equation for $v$ into the equation for $u$ in \eqref{0}, we obtain 
$u_t=\Delta u+u^2-\nabla v\cdot\nabla u$,
whose spatially homogeneous solutions are given by $u(t)=(T-t)^{-1}$. 
See Remark~\ref{remTypeI} below for a review of known results on type I / type II blow-up for system \eqref{0}.

Theorem~\ref{theo13}  is a consequence of the following more general result for type I solutions in any dimension $n\ge 3$.

\goodbreak

\begin{theo}\label{theo13b}
Let $n\ge 3$ and consider problem \eqref{0} with $\Omega=\R^n$ or $\Omega=B_R$, where $u_0$ satisfies \eqref{i0} and $T(u_0)<\infty$.
Assume 
   \be{hyp13.1}
   \hbox{Blow-up is of type~I}
   \ee
   and
    \be{hyp13.1b}
   \hbox{$B(u_0)\ne \R^n$ in case $\Omega=\R^n$.}
   \ee
Then the conclusions of Theorem~\ref{theo13} are true.  
\end{theo}

Theorem~\ref{theo13}  will be a consequence of Theorem~\ref{theo13b} and of the following type I blow-up result.

\begin{prop}\label{propA1}
Let $3\le n\le 9$ and consider problem \eqref{0} with $\Omega=\R^n$ or $\Omega=B_R$.
Assume that $u_0$ satisfies \eqref{i0}, \eqref{hyp13.1b} and $T(u_0)<\infty$.
Then blow-up is of type I.
\end{prop}

Proposition~\ref{propA1} in the case $\Omega=\R^n$ was proved in \cite{MS2}.
As for the case $\Omega=B_R$, Proposition~\ref{propA1} was proved in \cite{GP} 
but under restrictive additional assumptions involving intersection properties of $u_0$.
A~simpler proof covering both cases for general $u_0$ is given in Section~\ref{SecTypeI}
(see the discussion of methods in Section~\ref{ideas}).

\smallskip
On the other hand, we shall give partial extensions of the above results when the type I assumption is  replaced with the following condition on~$u_0$:
  \be{i2}
  	u_0\in C^1(\overline\Omega),\ \quad r^{n-1} u_{0,r}(r) + u_0(r) \int_0^r  u_0(s)s^{n-1} ds \ge 0,
	\ \mbox{ for all } r\in (0,R).
	\ee
Condition \eqref{i2} guarantees the time monotonicity of the averaged mass function (cf.~\eqref{w} below).
In this situation, although we cannot establish the sharp convergence  and profile as in Theorem~\ref{theo13}(i)(iii)
we can still obtain a satisfactory global lower bound on the space-time blow-up profile.

\begin{theo}\label{theo14}  Let $n\ge 3$ and consider problem \eqref{0}
with $\Omega=\R^n$ or $\Omega=B_R$.
Let $u_0$ satisfy \eqref{i0}, \eqref{i2} and $T(u_0)<\infty$.
 If $\Omega=\R^n$, assume in addition that $B(u_0)\ne \R^n$.
 Then there exists $C_1>0$ such that 
  \be{13.1b}
	 u(x,t) \ge  \frac{C_1}{u^{-1}(0,t)+|x|^2},
	\qquad (x,t)\in B_\rho\times[T/2,T),
  \ee
  with $\rho=R$ if $R<\infty$ and $\rho=1$ otherwise.
     \end{theo}

Combining with \cite[Theorem~1.3 and Remark~(vi)~p.670]{Soup-Win}, we obtain the two sided estimate
giving the qualitatively precise global description of the space-time blow-up behavior,
analogous to Theorem~\ref{theo13}(ii).

\begin{cor}
 Under the assumptions of Theorem~\ref{theo14}, there exist $C_1,C_2>0$, such that
\be{profileu0t}
\frac{C_1}{u^{-1}(0,t)+|x|^2}\le u(x,t) \le \frac{C_2}{u^{-1}(0,t)+|x|^2}
	\qquad (x,t)\in B_\rho\times[T/2,T).
\ee
\end{cor}

\section{Remarks and discussion}

\begin{rem} \label{remopen}
(i) {\bf(Blow-up set)} Assume \eqref{i0} with $\Omega=\R^n$. 
Then there holds either $B(u_0)=\{0\}$ or $B(u_0)=\R^n$.
Indeed, if some $\rho>0$ is not a blow-up point then, by integrating the first equation in \eqref{0}
over $B_\rho$ and using $u_r\le 0$, we obtain $\sup_{t\in (0,T)}\int_{B_\rho} u(t)\,dx<\infty$,
hence $u(x,t)\le C|x|^{-n}$ in $B_\rho\times(0,T)$, which yields $B(u_0)=\{0\}$.
When $u_0\in L^1(\R^n)$, the mass conservation property guarantees that $B(u_0)=\{0\}$.
When $\Omega=B_R$ we of course have $B(u_0)=\{0\}$ by the same argument.
Whether $B(u_0)=\R^n$ can occur unless $u_0$ is constant remains an open problem. 

\smallskip

(ii) {\bf(Time monotonicity)} For $n\ge 10$, it is still unknown whether the time monotonicity of the averaged mass function \eqref{w} 
 (guaranteed by condition \eqref{i2})
is sufficient to ensure type~I blow-up for system \eqref{0} under assumption \eqref{i0} (cf.~the discussion after \cite[Remark~1.4]{Naito-Senba-2012}).
The analogous property is true for the Fujita equation $u_t-\Delta u=u^p$ (cf.~\cite{FML} and \cite[Section 23]{quittner_souplet})  but the method does not seem to be easily extendable to \eqref{0} even in radial setting.

\smallskip
 
(iii) {\bf(Final profile)}  Assume that
\be{hypremopen3}
\Omega=\R^n, \hbox{ $u_0$ satisfies \eqref{i0}, and } B(u_0)\ne \R^n.
\ee
The result of \cite{BZ}, showing that either the sharp final profile behavior \eqref{finalblow-up} holds or $\lim_{x\to 0}|x|^2u(x,T)=0$,
is more generally valid under the condition that the limit 
\be{hypBZ}
\lim_{|x|\to\infty} |x|^{-n-2}\int_0^{|x|}  u_0(s) s^{n-1}ds\ \hbox{ exists in $[0,\infty]$}
\ee
(so that $u_0\in L^1(\R^n)$ is a special case).
Our Theorem~\ref{theo13b} shows that \eqref{finalblow-up} holds whenever blow-up is of type~I (so that $n\in[3,9]$ is a special case),
and also covers the case $\Omega_R$.
We do not know if \eqref{finalblow-up} holds whenever \eqref{hypremopen3} is true.
 \end{rem}

\begin{rem} \label{remTypeI}
{\bf(Type I and type II solutions)} 
(i) In dimensions $n\ge 3$, the self-similar or asymptotically self-similar type I solutions described above coexist with other kinds of solutions.
A class of solutions with type II blow-up has been constructed in \cite{Collot-2023}, making rigorous a formal argument from \cite{HMV}.
They constitute an open set of radial solutions in $L^\infty(\R^n)$ and are of the form (see \cite[Theorem~1.1]{Collot-2023}):
 \be{typeIIcosh}
u(x,t)= \frac{1}{\lambda^2(t)} \Big[W\Big(\frac{x-R(t)}{\lambda(t)}\Big)
+\bar u(x,t)\Big], \quad (x,t)\in\R^n\times (0,T), 
\ee
where $\ds\lim_{t\to T}\|\bar u(\cdot,t)\|_{L^\infty(\R^n)}=0$, 
$c_n=(n/2)^{1/n}$, $K=K(u_0)\in(0,\infty)$, and
$$W(\xi)=\frac18 cosh^{-2}(\xi/4),\quad
\lambda(t)\sim K(T-t)^{(n-1)/n},
\quad
R(t)\sim c_n K^{-1/n}(T-t)^{1/n},\quad\hbox{as } t\to T.$$
In particular their blow-up rate is given by $\|u(t)\|_\infty\sim (T-t)^{-2(n-1)/n}$.
Let us mention that a different class of type II blow-up radial solutions is also known to exist in dimensions $n\ge 11$ (see~\cite{MS}),
with a countable spectrum of blow-up rates of the form $\|u(t)\|_\infty\sim (T-t)^{-\gamma j}$ where $\gamma=\gamma(n)\ge 2/3$
and $j$ is an arbitrary integer $\ge 2$.
 On the other hand, in the case $n=3$ or $4$, radial solutions with a new, single-point type I blow-up mechanism
 have been constructed in \cite{Nguyen-Zaag-2023}.
 These solutions, obtained by a suitable linearization around the self-similar profile $U\equiv 1$,
satisfy:
 $$(T-t)u(y\sqrt{T-t},t)=1-\frac{\phi_{2\ell}(y)+o(1)}{B_\ell |\log (T-t)|}\quad \hbox{as }t\to T,$$
locally uniformly in $y$,
 where $\ell=n/(n-2)\in\{2,3\}$, $B_{\ell}$ is an integer, and the function $\phi_{2\ell}$ is a 
polynomial of degree $2\ell$ satisfying $\Delta \phi_{2\ell}-(1/2\ell)y\partial_y\phi_{2\ell}+\phi_{2\ell}=0$. 
  We stress that these solutions are radial but not decreasing in $|x|$,
  and are in sharp contrast with the result of Theorem~\ref{theo13}, 
  which rules out the occurence of the self-similar profile $U\equiv 1$ when $u_0$ is radially decreasing and $B(u_0)=\{0\}$.

\smallskip

(ii) The case $n=2$ is quite different, since any blow-up solution is of type II; see \cite{Naito-2008,Suzuki-2011,Miz22}.
These references, as well as \cite{HV,Vel,RS,Collot-2022}, contain information on the detailed blow-up asymptotics of solutions.
In particular, solutions of the form
$$
u(x,t)=\frac{1}{\lambda^2(t)} (U+\tilde u)\Big(\frac{x-x^*(t)}{\lambda(t)}\Big),\quad \hbox{as } (x,t)\to (0,T) 
$$
are constructed,
where the profile $U(\xi)=8 (1+|\xi|^2)^{-2}$ is a steady state,
$\lambda(t)\sim c\sqrt{T-t}\,e^{-\sqrt{|\log(T-t)|/2}}$, $\|\tilde u(\cdot,t)\|_{H^1(\R^2)}\to 0$, 
$x^*(t)\to 0$ as $t\to T$, and $c>0$ is an explicit universal constant.
This behavior, with blow-up rate $\|u(t)\|_\infty\sim (T-t)^{-1}e^{\sqrt{2|\log(T-t)|}}$,
 is stable (even in nonradial setting),
but other solutions with more singular (and unstable) type II blow-up rates 
are also constructed in \cite{Collot-2022}, with blow-up rates given by
 $\|u(t)\|_\infty\sim C_0(T-t)^{-j} |\ln(T-t)|^{-j/(j-1)}$ where $j$ is an arbitrary integer $\ge 2$.
Nonradial solutions involving the collision of two or more
 collapsing-rings have been constructed in 
\cite{SSV,Collot-2024,BDPM23}
 (see also \cite{BDPM25,HNS} for analogous results in the 3d-axisymmetric case).
The case $n=2$ also involves mass threshold and mass quantization phenomena that are absent for $n\ge 3$;
see \cite{Dolbeault-2004,BDP,BCM,BKLN,SuzBook,KaSou,Davila-2020,Ghoul-2018}.
In particular the final blow-up singularities are of Dirac mass type,
unlike the locally $L^1$ final profile 
of radially decreasing self-similar or type I solutions for $n\ge 3$~(cf.~\cite{Soup-Win}).
\end{rem}

\section{Ideas of proofs} \label{ideas}

As in previous work (see, e.g., \cite{BCKSV}, \cite{GP}), 
for a radial classical solution $(u,v)$ of \eqref{0}, we shall use the averaged mass function 
\be{w}
	w(r,t):=|B_r|^{-1} \int_{B_r}  u(x,t) \,dx=r^{-n}\int_0^r  u(s,t) s^{n-1}\,ds, 
	\qquad r\in (0,R), \ t\in [0,T)
\ee
(where $0<R\le\infty$), which satisfies a scalar parabolic equation
with quadratic reaction and critical convection terms (see \eqref{e1} below). 
Also we note that, by \eqref{w}, we have
\be{reluw} 
u=nw+rw_r.
\ee  
 We now describe the main ideas of proofs of our various results.

\runinsubsection{Theorems~\ref{theo13}-\ref{theo13b}.}
{\hskip -3mm}The starting point is the following microscopic version of 
 the convergence part of Theorem~\ref{theo13b}
(without radial monotonicity assumption).

\begin{prop}\label{propA2}
Let $n\ge 3$ and consider problem \eqref{0} with $\Omega=\R^n$ or $\Omega=B_R$, 
where $u_0\in L^\infty(\Omega)$, $u_0\ge 0$, is radially symmetric.
Assume that $T(u_0)<\infty$, blow-up is of type~I  and $B(u_0)=\{0\}$.
 
\smallskip
\noindent (i) Then there exists a nonnegative solution $\Psi\in C^2([0,\infty))$ of \eqref{eqnPsi} such that 
\be{cvpropA2w}
(T-t)w\big(y\sqrt{T-t},t\big)\underset{t\to T}{\longrightarrow} \Psi(y)\quad \hbox{in $C^1_{loc}(\R^n)$},
\ee
hence 
\be{cvpropA2}
(T-t)u\big(y\sqrt{T-t},t\big)\underset{t\to T}{\longrightarrow} U(y):=n\Psi (y)+y\Psi'(y)\quad 
\hbox{uniformly for $|y|$ bounded}.
\ee
 (ii) Assume in addition that $u_0$ is nonincreasing with respect to $|x|$,
and that $u_0$ is nonconstant in case $\Omega=\R^n$. Then $\Psi'<0$ in $(0,\infty)$,
hence in particular $U\not\equiv 0$ and $U\not\equiv 1$.
\end{prop}

In the case $\Omega=\R^n$, Proposition~\ref{propA2}
was proved in  \cite[Theorem~1.1]{GMS} 
 (except for the property  $U\not\equiv 1$ in the radial decreasing case).
In the case $\Omega=B_R$, this was proved in \cite{GP} under restrictive additional assumptions on the initial data
(which in particular force $w$ to select a limiting profile $\Psi$ with exactly one intersection with the singular solution 
$\psi^*(y)=2y^{-1}$ of \eqref{eqnPsi}). In Section~\ref{SecLocalCv}, we will give a simplified 
proof of Proposition~\ref{propA2} (and covering both $\Omega=\R^n$ and $\Omega=B_R$;
see at the end of this section for a sketch).

\smallskip 

With Proposition~\ref{propA2} at hand, the proof of the macroscopic self-similar behavior \eqref{13.1} in Theorem~\ref{theo13b} 
proceeds as follows. Denote by $V(x,t)=(T-t)^{-1}U(|x|/\sqrt{T-t})$ the self-similar solution of profile $U$.
For given, suitably small $x$, Proposition~\ref{propA2}  applied on a backward parabola $|x|=K\sqrt{T-t}$ 
with $K$ large, guarantees that $u/V$ is close to~$1$
up to the ``local time'' $t_0(x)=T-K^{-2}|x|^2$.
To propagate this information until the blow-up time, one combines the following observations:

\goodbreak

\begin{itemize}
\parskip=1pt
 \itemsep=1pt
 
\item[(a)]The remaining time interval $(t_0(x),T)$ has length $K^{-2}|x|^2\ll |x|^{-2}$

\item[(b)]$V$ is close to its known limiting profile $L|x|^{-2}$ on the interval $(t_0(x),T)$

\item[(c)]The time derivative enjoys a control of the form $|u_t|\le C|x|^{-4}$ near the singularity  (Proposition~\ref{lemdecaywt}).
The proof of this estimate relies on a nontrivial 
boostrap procedure. 
 \end{itemize}
 \noindent  Once \eqref{13.1} is proved,   \eqref{globalasympt} and \eqref{finalblow-up} can be derived as consequences of the
 properties of the profile~$U$.
 We note that the above strategy is inspired from an argument in \cite{MM09},
used there to determine the final space profile of blow-up
for the supercritical Fujita equation.
 The proof of the corresponding estimate for $u_t$ was shown in \cite{MM09} by an energy method.
We stress that such an argument cannot be implemented here due to the lack of variational structure.

\begin{rem} \label{remBZ}
Although both methods use zero number (ours does so through Proposition~\ref{propA2}), 
the argument used in \cite{BZ} to show
 that either \eqref{finalblow-up} holds or $\lim_{x\to 0}|x|^2u(x,T)=0$ is quite different.
It is based on the construction of partial self-similar solutions,
defined on the exterior of a parabola $|x|\ge c_0\sqrt{T-t}$, and on intersection-comparison with these solutions
(which requires assumption \eqref{hypBZ} to prevent zeros coming from infinity).
Although it does not rule out the possibility of a zero limit nor provides information on the space-time blow-up behavior of~$u$,
as an advantage, it can apply to type II solutions.
\end{rem}

\runinsubsection{Theorem~\ref{theo14}.}
{\hskip -3mm}We first show that the averaged mass function satisfies
\be{estimwm2}
w(x,t)\ge  \ts\frac12 w(0,t)
\ee
 on a generalized ``parabola'' $|x|\le cw^{-1/2}(0,t)$ for some $c>0$.
This is done by using the estimate $|w_r|\le Cw^{3/2}(0,t)$, which is a consequence of a known energy argument.

For given, suitably small $x$, this guarantees \eqref{estimwm2} for $t\le t_1(x)$, where 
the new local time $t_1(x)$ is defined by the implicit relation $|x|^2=cw^{-1}(0,t_1(x))$.
On the other hand, since assumption  \eqref{i2} ensures the monotonicity property $w_t\ge 0$, we have
$w(x,t)\ge w(x,t_1(x))= C|x|^{-2}$ on $(t_1(x),T)$, and it follows that
$w(x,t)\ge C\big(w^{-1}(0,t)+|x|^2)^{-1}$ in either time interval.
This estimate can finally transferred to the original function $u$ by an averaging procedure (Lemma~\ref{lem-monot2}),
making also use of the analogous upper estimate on $w$ from \cite{Soup-Win}.

\runinsubsection{Global lower space-time estimate.}
{\hskip -3mm}A variant of the argument of proof of Theorem~\ref{theo14} provides a direct proof (see after Theorem~\ref{theo13b2}) of the lower part of the estimate
\be{twosidedintro}
\frac{C_1}{T-t+|x|^2}\le u(x,t) \le \frac{C_2}{T-t+|x|^2}
\ee
for type I solutions. Estimate \eqref{twosidedintro} gives a qualitatively precise global description of the space-time blow-up behavior
 (the upper part was established in \cite{Soup-Win}), and
 is one of the conclusions of Theorems~\ref{theo13}-\ref{theo13b}, 
whose proof depends, among other things, on the microscopic convergence result in Proposition~\ref{propA2}.
The proof of the latter relies on intersection-comparison arguments
which are delicate and non robust, requiring the exact form of the equation. 
Although \eqref{twosidedintro} is less precise than Theorems~\ref{theo13}-\ref{theo13b},
it is thus interesting to have an alternative (and simpler) approach for such lower estimate,
which might be useful in different chemotaxis problems.

Namely, for type I solutions, replacing the energy argument by parabolic regularity applied on the equation
transformed by similarity variables (cf.~\eqref{VraieE0}), we obtain \eqref{estimwm2}  
up to the local time $t_2(x)=T-k|x|^2$ for some $k>0$. 
Since, unlike in the proof of Theorem~\ref{theo13b} based on Proposition~\ref{propA2}, 
the constant~$k$ cannot be taken arbitrarily small, the time-derivative control does not allow to propagate the estimate to
the remaining time interval $(t_2(x),T)$.
But the latter can be done by a suitable, local comparison argument on the equation for $w$,
and the estimate for $w$ is finally transferred to the original function $u$ by the above averaging procedure.

We note that this proof (as well as that of the upper estimate in \cite{Soup-Win}, based on elaborate auxiliary functionals), 
is actually independent of the existence and properties of self-similar solutions.

\runinsubsection{Proposition~\ref{propA1}.}
{\hskip -3mm}It was proved in \cite{MS2} in the case $\Omega=\R^n$ 
by means of braid group theory. An alternative, purely analytic proof,
making use of intersection-comparison techniques from~\cite{MM04} for Fujita type equations,
was also given in \cite{MS2},
but under the additional assumption $u_0(x)\le C(1+|x|)^{-1}$.
As for the case $\Omega=B_R$, Proposition~\ref{propA1} was proved in \cite{GP} 
under restrictive assumptions involving intersection properties of~$u_0$.
In Section~\ref{SecTypeI}, we shall provide a significantly simpler and purely analytic proof covering general $u_0$ in both cases,
relying on intersection-comparison arguments from \cite{CFG}
(see also \cite[Theorem 23.10]{quittner_souplet}),
used there for Fujita type equations in a ball and
simplifying those in~\cite{MM04}. 
To cover the case $\Omega=\R^n$, combined with these techniques, we also use 
backward uniqueness for linear parabolic equations like in \cite{MS2}.
However, owing to an additional argument using the decay of any radial steady state 
(see the proof of Lemma~\ref{lemA0} and Remark~\ref{remBwdU}),
backward uniqueness  is required only at the level of the equation for the averaged mass function $w$,
and not on the equation for $w_t$,
which enables us to avoid decay assumptions on~$u_0$.

\runinsubsection{Proposition~\ref{propA2}.}
{\hskip -3mm}It is proved by a dynamical systems argument divided in three steps.

\vskip 2pt

(a) First, as in \cite{GK, GMS}, one rewrites $w$ in similarity variables, setting:
$$\phi (y,s)=(T-t)w\big(y\sqrt{T-t},t\big),\quad y=\ts\frac{r}{\sqrt{T-t}},\ s=-\log (T-t).$$
The new function $\phi$ becomes a global solution of a modified parabolic equation (cf.~\eqref{VraieE0}),
whose steady-states are precisely the self-similar profiles $\Psi$
(i.e.,~solutions of \eqref{eqnPsi}),
and the type I assumption guarantees that $\phi$ is bounded.

\vskip 2pt

(b) Next, one derives a non-oscillation property at the origin, namely $\phi(0,s)$ has a limit as $s\to\infty$.
For this we follow the argument of \cite{GMS}, based on intersection-comparison with self-similar solutions,
making use of asymptotic properties of their profiles.

\vskip 2pt

(c) A key last step is to show that the $\omega$-limit set $\omega(\phi)$ consists of steady states
(which will conclude the proof, since is $\omega(\phi)$ will then be a singleton in view of Step~(b)).
In \cite{GMS} this is done by using delicate space-analyticity estimates for parabolic equations
with analytic nonlinearities.
Here we give a significantly simpler proof, relying instead on a 
nondegeneracy property (see Lemma~\ref{lemA5} and the paragraph after \eqref{limeqn})
for radial solutions of linear parabolic equations
(proved by a zero-number argument), 
applied to the equation for $\phi_s$.
 
\bigskip

The outline of the rest of the paper is as follows. 
Section~\ref{secprelim} gives some preliminaries on local well-posedness and on the averaged mass function.
Theorems~\ref{theo13}-\ref{theo13b} are proved in Section~\ref{secmacro}.
Section~\ref{secmacro2} is devoted to the proof of Theorem~\ref{theo14} and to the simpler direct proof of 
the space-time lower bound.
Propositions~\ref{propA1} and \ref{propA2} on type I blow-up
and local convergence in similarity variables are proved in Sections~\ref{sectype1} and \ref{SecLocalCv}, respectively.

\goodbreak

\mysection{Preliminaries: local well-posedness and averaged mass}
\label{secprelim}

Recalling definition \eqref{w}, hence
$w(r,t) =\int_0^1 u(r\sigma,t)\,d\sigma$, we see that
$w\in C^{2,1}([0,R)\times(0,T))$,
and direct calculation (see, e.g.,~\cite[Proposition~3.2]{Soup-Win}) shows that $w$ is a classical solution of
\be{e1}
        \left\{ 
        \begin{array}{lcll}
    	\hfill w_t - w_{rr}-\frac{n+1}{r}w_r&=&(nw+rw_r)w, 
	    	& r\in (0,R), \ t\in (0,T), \\ [1mm]
    	\hfill w_r&=&0,
    	& r=0,\ t\in (0,T), \\ [1mm]
	    	\hfill w&=&\mu, 
    	& r=R,\ t\in (0,T)\quad\hbox{(if $R<\infty$)},
	 \end{array}
	\right.
\ee
where $\mu=R^{-n}\|u_0\|_1$, 
and the last condition follows from the mass conservation property for $u$
(cf.~Proposition~\ref{locprop}(iii) below).
On the other hand, setting $\tilde\Omega=\R^{n+2}$ or 
$\tilde\Omega=\{\tilde x\in\R^{n+2};\ |\tilde x|\le R\}$, and writing
$w(\tilde x,t)=w(|\tilde x|,t)$, for $(\tilde x,t)\in \tilde\Omega\times (0,T)$, 
we can also treat $w$ as a solution of 
\be{highereq}
        \left\{ 
        \begin{array}{lcll}
    	\hfill w_t-\tilde\Delta w&=&(nw+\tilde x\cdot\tilde\nabla w)w \equiv u w
    	& \hbox{ in $\tilde\Omega \times (0,T)$}, \\ [1mm]
    	    	\hfill w&=&\mu, 
    	& \hbox{ on $\partial\tilde\Omega \times (0,T)$},
	 \end{array}
	\right.
\ee
where  $\tilde\Delta$, $\tilde\nabla$ are respectively the Laplacian and the spatial gradient in $n+2$ space variables.

The solution of \eqref{0} considered in this paper is given by the 
following local existence-uniqueness result (where $S(t)$ is the heat semigroup on $L^\infty(\R^n)$).

\begin{prop} \label{locprop}
Let  $n\ge 2$ and consider problem \eqref{0} with $\Omega=\R^n$ or $\Omega=B_R$, where $u_0\in L^\infty(\Omega)$, $u_0\ge 0$, 
is radially symmetric.
\begin{itemize}
\parskip=1pt
 \itemsep=1pt
\item[(i)]
There exists $\tau>0$ and a unique, classical solution $(u,v)$ of \eqref{0} such that
 $(u,v)\in BC^{2,1}(\overline\Omega\times (0,\tau))\times BC^{2,0}(\overline\Omega\times (0,\tau))$,
 $u\in L^\infty_{loc}([0,T);L^\infty(\Omega))$,
 and the initial data is taken in the sense
$\lim_{t\to 0}\|u(t)-u_0\|_q=0$ for all $q\in[1,\infty)$ if $\Omega=B_R$, 
 or $\lim_{t\to 0}\|u(t)-S(t)u_0\|_\infty=0$ if $\Omega=\R^n$.
  Moreover, $(u,v)$ can be extended to a unique maximal solution,
  whose existence time $T=T(u_0,v_0)\in (0,\infty]$ satisfies
  \be{ext}
	\mbox{either \ $T<\infty$ \ or \ $\displaystyle\lim_{t\to T} \|u(\cdot,t)\|_\infty = \infty$.}
  \ee
Also, the function $w$ in \eqref{w} solves \eqref{e1} and $u, w$
are radially symmetric.

\item[(ii)] If $u_0$ is nonincreasing with respect to $|x|$, then 
\be{e1wr}
u_r, w_r\le 0.
\ee

 \item[(iii)] Assume that $u_0\in L^1(\R^n)$ in case $\Omega=\R^n$. Then $u$ enjoys the mass conservation property
 \be{mass}
\|u(t)\|_{L^1(\Omega)}=\|u_0\|_{L^1(\Omega)},\quad 0<t<T.
\ee
\end{itemize}
\end{prop}

See the appendix for a proof.
The local existence-uniqueness part is essentially known. 
As for assertion (ii) since a complete proof seems to be missing in the literature,
and is somehow delicate especially in the case $\Omega=\R^n$
or without additional regularity on the initial data, we provide one.

\section{Sharp macroscopic blow-up behavior: Proof of Theorems~\ref{theo13}-\ref{theo13b}}
\label{secmacro}

For the proof of Theorem~\ref{theo13b}, as explained in Section~\ref{ideas},
we need the following estimate of $|u_t|$ near the singularity. 

\begin{prop}\label{lemdecaywt} 
Let  $n\ge 3$ and consider problem \eqref{0} with $\Omega=\R^n$ or $\Omega=B_R$, where $u_0\in L^\infty(\Omega)$, $u_0\ge 0$, 
is radially symmetric.
Assume that
\be{hyp-lemdecaywt} 
u(x,t)\le C_1|x|^{-2}\quad\hbox{in $(B_\sigma\setminus\{0\})\times(0,T)$,}
\ee
for some $C_1,\sigma>0$.  
Then there exist $\eta,C_2>0$ such that 
\be{decaywt}
|u_t(x,t)| \le C_2|x|^{-4}\quad \hbox{in } (B_\eta\setminus\{0\})\times (T/2,T).
\ee
\end{prop}

To prove Proposition~\ref{lemdecaywt}, we use a bootstrap argument on the elliptic and parabolic parts of~\eqref{0}, along with 
the following 
scaled parabolic imbedding and scaled versions of the $L^q$ and Schauder interior parabolic regularity estimates.
In what follows, for given $(\xi,\tau)\in\R^n\times\R$, we set $Q_\rho:=B_\rho(\xi)\times (\tau-\rho^2,\tau)$, 
and denote the 
{parabolic H\"older bracket by
$$[f]_{\alpha,Q_\rho}:=\ds\sup_{(x,t)\ne (x',t')\in Q_\rho} \ts\frac{|f(x,t)-f(x',t')|}{|x-x'|^\alpha+|t-t'|^{\alpha/2}}.$$

\begin{lem}\label{lemdecaywt00} 
Let  $\alpha\in(0,1)$, $\max\big(n+2,\frac{n+2}{2(1-\alpha)}\big)<q<\infty$,  
$(\xi,\tau)\in\R^n\times\R$ and $\rho>0$. 
There exists a constant $C_0>0$ depending only on $n,q,\alpha$ such that, 
for all $z\in W^{2,1;q}(Q_\rho)$, 
\be{estimat1.3}
\rho^{\alpha} [z]_{\alpha,Q_\rho}+\rho \|\nabla z\|_{L^\infty(Q_\rho)}
\le C_0\rho^{-(n+2)/q} \left(\rho^2\|D^2z\|_{L^q(Q_\rho)} + \rho^2\|z_t\|_{L^q(Q_\rho)}+\|z\|_{L^q(Q_\rho)}\right).
\ee
\end{lem}

\begin{lem}\label{lemdecaywt0} 
Let $\alpha\in(0,1)$, $q\in(1,\infty)$, $(\xi,\tau)\in\R^n\times\R$, $\rho, A>0$,  
$b\in L^\infty(Q_\rho)$, and assume $\|b\|_{L^\infty(Q_\rho)}\le A\rho^{-1}$.
There exists a constant $C_A>0$ depending only on $A,n,q,\alpha$ with the following properties. 
\smallskip

(i) If $z$ is a strong solution of 
\be{neweqnu}
z_t-\Delta z+b\cdot\nabla z=f\quad \hbox{in } \ Q_\rho,
\ee
then
\be{estimat1.1}
\|D^2z\|_{L^q(Q_{\rho/2})} + \|z_t\|_{L^q(Q_{\rho/2})}
\le C_A\Big(\|f\|_{L^q({Q_\rho})}+\rho^{-2}\|z\|_{L^q(Q_\rho)}\Big).
\ee

(ii)
If $z$ is a classical solution of \eqref{neweqnu} and $[b]_{\alpha,Q_\rho}\le A\rho^{-1-\alpha}$, then 
\be{estimat1.2}
\begin{aligned}
	\rho^{\alpha} [D^2z]_{\alpha,Q_{\rho/2}} +\rho^{\alpha} [z_t]_{\alpha,Q_{\rho/2}}
	&+\|D^2z\|_{L^\infty(Q_{\rho/2})} + \|z_t\|_{L^\infty(Q_{\rho/2})} \\
	&\le C_A\Big(\rho^{\alpha}[f]_{\alpha,Q_\rho}+\|f\|_{L^\infty(Q_\rho)}+\rho^{-2}\|z\|_{L^\infty(Q_\rho)}\Big).
\end{aligned}
\ee
\end{lem}

\begin{proof}[Proof of Lemma~\ref{lemdecaywt00}] 
Let 
\be{deltildez}
\tilde z(y,s):=z(\xi+\rho y,\tau+\rho^2s),\quad (y,s)\in Q=(-1,0)\times B_1.
\ee
By the standard imbeddings \cite{LSU}, we have
$$[\tilde z]_{\alpha,Q}+\|\nabla \tilde z\|_{L^\infty(Q)}
\le C_0 \left(\|D^2\tilde z\|_{L^q(Q)} + \|\tilde z_t\|_{L^q(Q)}+\|\tilde z\|_{L^q(Q)}\right), 
$$
and \eqref{estimat1.3} follows by scaling back to $z$. 
\end{proof}

\begin{proof}[Proof of Lemma~\ref{lemdecaywt0}] 
The function $\tilde z$ defined in \eqref{deltildez} 
satisfies $\tilde z_t-\Delta \tilde z+\tilde b\cdot\nabla \tilde z=\tilde f:=\rho^2 f(\xi+\rho y,\tau+\rho^2s)$ in $B_1\times(-1,0)$
with $\tilde b(y,s)=\rho b(\xi+\rho y,\tau+\rho^2s)$.
By assumption we have $[\tilde b]_{\alpha,Q_\rho}\le A$ (and $\|\tilde b\|_{L^\infty(Q_\rho)}\le A$ in assertion (ii)).
Consequently, by the usual parabolic estimates (see \cite{Lieb}, \cite[Section~48.1]{quittner_souplet}), 
inequalities \eqref{estimat1.1} and \eqref{estimat1.2} for $\rho=1$ are true with $z, f$ replaced by $\tilde z,\tilde f$.
 The conclusion then follows by scaling back to $z$. \end{proof}

\begin{proof}[Proof of Proposition~\ref{lemdecaywt}] 
In this proof $C$ will denote various positive constants, which may change from line to line
and may depend on the solution $(u,v)$. 

\smallskip

$\bullet$ 
We have
\be{decaywt-e1}
|\nabla v|\le C|x|^{-1}\quad\hbox{in $(B_\sigma\setminus\{0\})\times(0,T)$.}
\ee
Indeed, for $r\in(0,\sigma)$, rewriting the second equation in \eqref{0} as $-(r^{n-1}v_r)_r=r^{n-1}u$,
 and using \eqref{hyp-lemdecaywt}, we have
$$|v_r(r,t)|=r^{1-n}\int_0^r s^{n-1}u(s,t)\,ds\le Cr^{1-n}\int_0^r s^{n-3}\,ds\le Cr^{-1},$$
owing to $n\ge 3$, hence \eqref{decaywt-e1}. 

\smallskip
$\bullet$ 
Pick $\eta\in(0,\min(\sigma/2,\sqrt{T/4}))$.
Let $\xi\in B_\eta$, $\tau\in(T/4,T)$ and set $\rho=|\xi|/2$.
We claim that
\be{decaywt-e2}
[u]_{\alpha,Q_\rho}\le C\rho^{-2-\alpha} 
\ee
and
\be{decaywt-e2a}
|\nabla u|\le C|x|^{-3}\quad \hbox{in } (B_\eta\setminus\{0\})\times (T/2,T).
\ee
By the first equation in  \eqref{0}, $u$ is a solution of \eqref{neweqnu}
with $b=\nabla v$ and $f=u^2$.
Since $\|b\|_{L^\infty(Q_\rho)}\le M\rho^{-1}$ by \eqref{decaywt-e1}, we may apply the parabolic $L^q$ estimate \eqref{estimat1.1}, also using \eqref{hyp-lemdecaywt}, to deduce
$$\|D^2u\|_{L^q(Q_{\rho/2})} + \|u_t\|_{L^q(Q_{\rho/2})}
\le C\Big(\|u^2\|_{L^q({Q_\rho})}+\rho^{-2}\|u\|_{L^q(Q_\rho)}\Big)\le C\rho^{-4+(n+2)/q}.
$$
By the imbedding \eqref{estimat1.3} and \eqref{hyp-lemdecaywt}, we get
$$\rho^{\alpha} [u]_{\alpha,Q_\rho}+\rho \|\nabla u\|_{L^\infty(Q_\rho)}
\le C\rho^{-(n+2)/q} \left(\rho^2\|D^2u\|_{L^q(Q_\rho)} + \rho^2\|u_t\|_{L^q(Q_\rho)}+\|u\|_{L^q(Q_\rho)}\right)\le C\rho^{-2},
$$
hence \eqref{decaywt-e2} and  \eqref{decaywt-e2a}.

\smallskip
$\bullet$ Let 
\be{defxirho}
\xi\in B_\eta,\quad \tau\in(T/2,T),\quad \rho:=|\xi|/2.
\ee
We claim that
\be{decaywt-e3}
[\nabla v]_{\alpha,Q_\rho}\le C\rho^{-1-\alpha}.
\ee
By the second equation in \eqref{0} and \eqref{decaywt-e1}, we have  
\be{decaywt-e4}
|v_{rr}|\le \frac{n-1}{r}|v_r|+u\le C|x|^{-2} \quad\hbox{in $(B_\sigma\setminus\{0\})\times(0,T)$.}
\ee
On the other hand, for $r\in(0,\eta)$ and $T/2<t<t'<T$, \eqref{decaywt-e2} implies
$|u(r,t)-u(r,t')|\le Cr^{-2-\alpha}|t-t'|^{\alpha/2}$. Consequently,
$$\begin{aligned}
|v_r(r,t)-v_r(r,t')|
&\le r^{1-n}\int_0^r s^{n-1}|u(s,t)-u(s,t')|\,ds\\
&\le C|t-t'|^{\alpha/2}r^{1-n}\int_0^r s^{n-3-\alpha}\,ds\le C|t-t'|^{\alpha/2} r^{-1-\alpha},
\end{aligned}$$
owing to $n\ge 3$ and $\alpha<1$.
By combining this with \eqref{decaywt-e4} we easily deduce \eqref{decaywt-e3}.

\smallskip
$\bullet$ Now, for $\xi\, \tau, \rho$ as in \eqref{defxirho}, since $u$ is a solution of \eqref{neweqnu}
with $b=\nabla v$, $f=u^2$ and, in view of \eqref{decaywt-e1} and \eqref{decaywt-e3}, the coefficient $b$ satisfies the assumptions of Lemma~\ref{lemdecaywt0}(ii),
 we may thus apply the parabolic Schauder estimate \eqref{estimat1.2}, along with \eqref{hyp-lemdecaywt}, \eqref{decaywt-e2} and
 $[u^2]_{\alpha,Q_\rho}\le 2\|u\|_{L^\infty(Q_\rho)}[u]_{\alpha,Q_\rho}$, to deduce
$$
\begin{aligned}
 \|u_t\|_{L^\infty(Q_{\rho/2})} 
	&\le C\Big(\rho^{\alpha}[u^2]_{\alpha,Q_\rho}+\|u^2\|_{L^\infty(Q_\rho)}+\rho^{-2}\|u\|_{L^\infty(Q_\rho)}\Big)\le C\rho^{-4},
	\end{aligned}
$$
hence \eqref{decaywt}.
\end{proof}

\begin{proof}[Proof of Theorem~\ref{theo13b}](i)
We put $R=\infty$ in case $\Omega=\R^n$.
Let $\Psi, U$ be given by Proposition~\ref{propA2}(ii). We set
\be{defW}
V(x,t)=\frac{1}{T-t}U\left(\frac{|x|}{\sqrt{T-t}}\right),\quad (x,t)\in \R^n\times(-\infty,T).
\ee
Fix $\eps\in(0,1/2)$. From \cite[Proposition~3.1 and Lemma~3.7]{GMS}, we have
$
\lim_{y\to\infty} y^2\Psi(y)=L_0\in (0,\infty)
$
and
$
\lim_{y\to\infty} \frac{y\Psi'(y)}{\Psi(y)}=-2,
$
hence
\be{GMS1}
\lim_{y\to\infty} y^2U(y)=L:=(n-2)L_0.
\ee
We also note that 
\be{posU}
\hbox{$U>0$ on $[0,\infty)$.}
\ee
Indeed, since $U>0$ at infinity by \eqref{GMS1},
otherwise there would exist $y>0$ such that 
$U(y)=y \Psi_y(y)+n\Psi(y)=0$ and
 $U'(y)=y\Psi_{yy}(y)+(n+1)\Psi_y(y)\ge 0$. 
But, by~\eqref{eqnPsi}, this would imply
$-(n-2)\Psi(y)=y\Psi_y(y)+2\Psi(y)=2(\Psi_{yy}(y)+\frac{n+1}{y}\Psi_y(y))\ge 0$: a contradiction. 

Next, by \eqref{GMS1} there exists $K_\eps\ge 1$ such that
\be{psi1}
(1-\eps)L|y|^{-2}\le U(y)\le (1+\eps)L|y|^{-2} ,\quad |y|\ge K_\eps.
\ee
In particular, we have
\be{psi2}
(1-\eps)L|x|^{-2}\le V(x,t)\le (1+\eps)L|x|^{-2},\quad x\in \R^n,\ t\ge T-K_\eps^{-2}|x|^2.
\ee
Let $K=\max(\eps^{-1/2},K_\eps)$.
By Proposition~\ref{propA2} and \eqref{posU}, 
there exists $\sigma_\eps\in(0,T)$ such that 
\be{cvg1}
(1-\eps)V(x,t)\le u(x,t) \le (1+\eps)V(x,t),\quad  T-\sigma_\eps\le t<T,\ |x|\le K\sqrt{T-t}.
\ee

Now fix any $x\in \R^n$ with $|x|<R_K:=\min(R/2, K(\sigma_{\eps})^{1/2})$, set 
\be{deft0}
t_0(x):=T-K^{-2}|x|^2\ge T-\sigma_\eps,
\ee
and let $t_1\in(t_0(x),T)$. Integrating \eqref{decaywt}  in Proposition~\ref{lemdecaywt} 
over $(t_0(x),t_1)$, it follows that 
\be{estimt0u}
|u(x,t_1)-u(x,t_0(x))|\le \int_{t_0(x)}^{t_1}|u_t(x,\tau)|d\tau\le C|x|^{-4}(T-t_0(x))\le CK^{-2}|x|^{-2}
\le C\eps |x|^{-2}.
\ee
Since estimate \eqref{decaywt} holds for $V$ as well, we also get
\be{estimt0V}
|V(x,t_1)-V(x,t_0(x))|\le  C\eps |x|^{-2}.
\ee
By \eqref{estimt0u}, \eqref{estimt0V}, we have
$$\begin{aligned}
\left| u(x,t_1)-V(x,t_1)\right|
&\le \left| u(x,t_1)-u(x,t_0(x))\right|+\left| u(x,t_0(x))-V(x,t_0(x))\right|+\left| V(x,t_0(x))-V(x,t_1)\right|\\
&\le \left| u(x,t_0(x))-V(x,t_0(x))\right|+2C\eps |x|^{-2}.
\end{aligned}$$
Recalling \eqref{deft0} and using \eqref{cvg1} for $t=t_0(x)$, it follows that
$$\left| u(x,t_1)-V(x,t_1)\right|\le \eps(V(x,t_0(x))+2C|x|^{-2}).$$
Since, by \eqref{psi2}, $V(x,t_0(x))\le 3V(x,t_1)$
and $|x|^{-2}\le 2L^{-1}V(x,t_1)$, we deduce that
 $$\begin{aligned}
\left| u(x,t_1)-V(x,t_1)\right|
\le 3\eps(1+CL^{-1})V(x,t_1),\quad
|x|<R_K,\ t\in(t_0(x),T).
\end{aligned}$$
Combining this with \eqref{cvg1}, valid in 
the complementary range $t\in[T-\sigma_K,t_0(x)]$, we conclude that, for any $\tilde\eps\in(0,1/2)$,
exist $R_1\in(0,R)$ and $\tau\in(0,T)$ such that
\be{e1a}
(1-\tilde\eps)V(x,t)\le u(x,t)\le (1+\tilde\eps)V(x,t),\quad \hbox{for all  $|x|\le R_1$ and $t\in [T-\tau,T)$,}
\ee
which is the desired convergence.

\smallskip

(ii)  Let $V$ be defined by \eqref{defW}.
By \eqref{GMS1}, \eqref{posU}, there exist $c_1,c_2$ such that
$c_1\le (1+|y|^2)U(y)\le c_2$ for all $y\ge 0$,
hence
$$
\frac{c_1}{T-t+|x|^2}\le V(x,t) \le \frac{c_2}{T-t+|x|^2},
	\qquad (x,t)\in \R^n\times(0,T).
	$$
	Applying \eqref{13.1} in Theorem~\ref{theo13} and taking $\eta,\tau>0$ small enough so that $1/2\le \eps(x,t)\le 2$ for all $(x,t)\in B_\eta\times(T-\tau,T)$, it follows that
\be{globalasympt2}
\frac{c_1/2}{T-t+|x|^2}\le u(x,t) \le \frac{2c_2}{T-t+|x|^2},
	\qquad (x,t)\in B_\eta\times(T-\tau,T).
	\ee
	
Finally, we have $u>0$ in $\Omega\times(0,T)$ by the strong maximum principle.
	Also, since $u$ remains bounded outside of the origin, it extends by parabolic regularity to a function 
	$u\in C^{2,1}((\overline\Omega\setminus\{0\})\times[T/2,T])$. Consequently
	$u(\rho,t)\ge c>0$ for all $t\in [T/2,T]$ (with $\rho=1$ if $\Omega=\R^n$
	and $\rho=R$ in case $\Omega=B_R$, using Hopf's Lemma 
	and the boundary conditions in \eqref{0} for $u$ in the second case).
	In view of this and \eqref{globalasympt2}, we readily obtain \eqref{globalasympt}
	for some constants $C_1, C_2>0$.

\smallskip

(iii) By \eqref{GMS1}, for all $x\in\R^n\setminus\{0\}$, we have $\ds\lim_{t\to T}V(x,t)=L|x|^{-2}$.
For any $\delta\in(0,1/2)$, by \eqref{13.1} in Theorem~\ref{theo13}, there exist $\eta,\tau>0$ such that 
$$(1-\delta)V(x, t)\le u(x,t)\le (1+\delta)V(x, t), \quad (x,t)\in B_\eta\times(T-\tau,T).$$
Letting $t\to T$, we deduce 
that $(1-\delta)L|x|^{-2}\le u(x,T)\le (1+\delta)L|x|^{-2}$ for all $x\in B_\eta\setminus\{0\}$.
 The assertion follows.
\end{proof}

\begin{proof}[Proof of Theorem~\ref{theo13}]
This is a direct consequence of Theorem~\ref{theo13b}, Proposition~\ref{propA1} 
and the fact that $B(u_0)=\{0\}$ whenever $u_0$ satisfies \eqref{i0} and $u_0\in L^1(\R^n)$
(cf.~Remark~\ref{remopen}(i)). 
\end{proof}

\section{Proof of Theorem~\ref{theo14} and simpler direct proof of space-time lower bound}
\label{secmacro2}

In this section, we put $r=|x|$ and denote
$$m(t)=w(0,t).$$
Also we set $R=1$ in case $\Omega=\R^n$.
 The proof of Theorem~\ref{theo14}  relies on two lemmas.
The first one provides a lower space-time estimate of the averaged mass function $w$ 
under the monotonicity assumption \eqref{i2}.

\begin{lem}\label{lem-monot1}
Under the assumptions of Theorem~\ref{theo14}, there exist $c_1,\eta>0$ such that  
\be{estiminfwmonotA0}
w(x,t)\ge c_1\big(m^{-1}(t)+|x|^2\big)^{-1},\quad |x|\le \eta,\ 0\le t<T.
\ee
\end{lem}

\begin{proof}
First, there exists $C_1>0$ such that 
\be{gradest}
|w_r|\le C_1m^{3/2}(t),\quad 0\le |x|<R,\ 0< t<T 
\ee
(cf.~\cite{GP, SouProc, Soup-Win}). We recall the simple proof  for completeness.
By assumption \eqref{i2} and \cite[Proposition~3.5]{Soup-Win}), we have $w_t\geq 0$.
Since also $w_r\leq 0$ and $w\ge 0$, \eqref{e1} yields
$${\partial\over\partial r}\Bigl({1\over 2}w_r^2+{n\over 3}w^{3}\Bigr)
=(w_{rr}+nw^2)w_r
=\Bigl(w_t -\frac{n+1}{r}w_r-rww_r\Bigr)w_r\leq 0.$$
This guarantees
$$\Bigl({1\over 2}w_r^2+{n\over 3}w^3\Bigr)(r,t)
\leq {n\over 3}w^3(0,t),$$
hence \eqref{gradest}.

Next setting $C_2:=(2C_1)^{-1}$, it follows from \eqref{gradest} that
\be{pfmonot1}
w(x,t)\ge m(t)\big(1-C_1|x|m^{1/2}(t)\big)\ge \ts\frac12 m(t),\quad\hbox{for $|x|\le C_2m^{-1/2}(t)$ and $0< t<T$.} 
\ee
Fix $x_0$ with $|x_0|<\eta:=\min(R,C_2m^{-1/2}(0))$.
Since $\lim_{t\to T}m(t)=\infty$, there exists a minimal $t_1=t_1(x_0)\in(0,T)$ such that 
\be{pfmonot1b}
|x_0|=C_2m^{-1/2}(t_1(x_0)).
\ee
Consequently we have $|x_0|\le C_2m^{-1/2}(t)$ for $t\in[0,t_1(x_0)]$ and \eqref{pfmonot1} implies
\be{pfmonot2}
w(x_0,t)\ge  \ts\frac12 m(t),\quad 0\le t\le t_1(x_0).
\ee
On the other hand, since $w_t\ge 0$, it follows from \eqref{pfmonot1b} and \eqref{pfmonot2} that 
\be{pfmonot3}
w(x_0,t)\ge w(x_0,t_1(x_0))\ge\ts\frac12 m(t_0)=C_3|x_0|^{-2},\quad t_1(x_0)<t<T, 
\ee
where $C_3=\frac12 C_2^2=(8C_1^2)^{-1}$. Setting $c_1=\min(C_3,\frac12)$ 
and combining \eqref{pfmonot2} and \eqref{pfmonot3}, we obtain
$$w(x_0,t)\ge c_1\big(m^{-1}(t)+|x_0|^2\big)^{-1},\quad 0\le t<T,$$
hence \eqref{estiminfwmonotA0}.
\end{proof}

Our second lemma, based on an averaging procedure, provides a lower space-time estimate of $u$ assuming lower and upper estimates of $w$.

\begin{lem}\label{lem-monot2}
Let $n\ge 3$ and consider problem \eqref{0}
with $\Omega=\R^n$ or $\Omega=B_R$. Let $u_0$ satisfy \eqref{i0} and $T(u_0)<\infty$
and assume that there exist $c_1,c_2,\eta>0$ 
and $\tau\in(0,T)$ such that 
\be{estiminfwmonotA}
c_1\big(m^{-1}(t)+|x|^2\big)^{-1}\le w(x,t)\le c_2\big(m^{-1}(t)+|x|^2\big)^{-1},\quad |x|\le \eta,\ \tau\le t<T.
\ee
Then there exists $c_3>0$ such that 
  \be{13.1b2}
	 u(x,t) \ge  c_3\big(u^{-1}(0,t)+|x|^2\big)^{-1},
	\qquad (x,t)\in B_R\times[T/2,T).
  \ee
\end{lem}

\smallskip
\begin{proof}
 It suffices to prove \eqref{13.1b2} with $\tau$ in place of $T/2$ and some $R_1\in(0,R)$
(arguing as in the second paragraph of the proof of Theorem~\ref{theo13b}(ii)). 
Fix $t\in[\tau,T)$ and set $\sigma=1/m(t)$.
 For $r\in[0,\eta]$,
by the lower part of \eqref{estiminfwmonotA} and \eqref{w}, we have
\be{e10}
\int_0^r s^{n-1}u(s,t)\,ds=r^n w(r,t) \ge \frac{c_1r^{n}}{\sigma+ r^2}.
\ee
By the upper part of \eqref{estiminfwmonotA} and \eqref{reluw} 
we have $u(r,t)\le nw(r,t)\le C_3(\sigma+r^2)^{-1}$
with $C_3=nc_2$.
 Observe that the function $s\mapsto \frac{s^{n-1}}{\sigma+s^2}$ is nondecreasing owing to $n\ge 3$.
It follows that
\be{e11}
\int_0^r s^{n-1}u(s,t)\,ds\le
 C_3\int_0^r  \frac{s^{n-1}}{\sigma+s^2}\,ds\le\frac{C_3r^n}{\sigma+r^2}.
\ee
 Let $K>1$ to be chosen below and assume $r\le \eta/K$.
 Since $u_r\le 0$, combining (\ref{e10}),  where $r$ is replaced by $Kr$, with (\ref{e11}), we deduce that
\be{e11b}
\begin{array}{ll}	
\displaystyle{(Kr)^nu(r,t)\over n}
&\ge \displaystyle\int_r^{Kr} s^{n-1}u(s,t) \,ds= \int_0^{Kr} s^{n-1}u(s,t) \,ds- \int_0^r s^{n-1}u(s,t)\,ds \\ [3mm]
&\ge \displaystyle\frac{c_1K^nr^{n}}{\sigma+K^2r^2}-\frac{C_3r^n}{\sigma+r^2}.
\end{array}
\ee
 Taking $K=1+(2C_3/c_1)^{1/(n-2)}$, we get
$$\frac{c_1K^nr^{n}}{\sigma+K^2r^2}=c_1K^{n-2}\frac{K^2r^{n}}{\sigma+K^2r^2}
\ge 2C_3\frac{r^n}{\sigma+r^2}.$$
It then follows from \eqref{e11b} that
$$\displaystyle{(Kr)^nu(r,t)\over n}\ge \frac{C_3r^n}{\sigma+r^2}.$$
Since, by \eqref{reluw},
$u(0,t)=nm(t)$, 
we obtain
$$u(r,t)\ge \frac{nC_3K^{-n}}{\sigma+r^2}\ge \frac{C_3K^{-n}}{u^{-1}(0,t)+r^2},
\quad r\le \eta/K,\ \tau\le t<T,$$ 
hence \eqref{13.1b2}.
\end{proof}

\begin{proof}[Proof of Theorem~\ref{theo14}]
The lower and upper parts of \eqref{estiminfwmonotA} are guaranteed by
Lemma~\ref{lem-monot1} and by
\cite[Theorem~1.3 and Remark~(vi)~p.670]{Soup-Win}, respectively.
The conclusion then follows from Lemma~\ref{lem-monot2}.
\end{proof}

The following two-sided estimate  (whose upper part is from \cite{Soup-Win})
gives  the qualitatively precise global description of the space-time blow-up behavior
for type I solutions.
It is contained in Theorem~\ref{theo13b}. 

\begin{theo}\label{theo13b2}
Let $n\ge 3$ and consider problem \eqref{0} with $\Omega=\R^n$ or $\Omega=B_R$, where $u_0$ satisfies \eqref{i0} and $T(u_0)<\infty$.
Assume
\eqref{hyp13.1}, \eqref{hyp13.1b}.
Then there exist $C_1, C_2>0$ such that 
$$ 
\frac{C_1}{T-t+|x|^2}\le u(x,t) \le \frac{C_2}{T-t+|x|^2},
	\qquad (x,t)\in B_R\times[T/2,T).
$$
\end{theo}

As explained in Section~\ref{ideas},
we now give a direct and simpler proof of the lower estimate in Theorem~\ref{theo13b2}, 
independent of the microscopic convergence result from Proposition~\ref{propA2}.

\begin{proof}[Proof of  the lower estimate in Theorem~\ref{theo13b2}]
 Integrating the inequality $m'(t)=w_t(0,t)\le nm^2(t)$ 
(which follows from \eqref{e1}), we get $m(t)\ge n^{-1}(T-t)^{-1}$. Using
$u(0,t)=nw(0,t)$, due to \eqref{reluw}, and assumption \eqref{hyp13.1}, we deduce that
\be{pfmonotB0b}
 (T-t)^{-1}\le u(0,t)\le M(T-t)^{-1}\quad\hbox{and}\quad
m(t)=n^{-1}u(0,t), \quad  0\le t<T.
\ee
In view of  Lemma~\ref{lem-monot2} and \eqref{pfmonotB0b},
since the upper estimate in \eqref{estiminfwmonotA} is guaranteed by \cite{Soup-Win}, it suffices to show that, for some $\sigma\in(0,R]$ and $\tau\in(0,T)$,
\be{estiminfwmonot}
w(x,t)\ge C\big(T-t+|x|^2\big)^{-1},\quad |x|<\sigma,\ \tau\le t<T. 
\ee

We first claim that 
\be{pfmonotB0} 
|w_r|\le M_1(T-t)^{-3/2}, \quad 0\le |x|<\bar\rho\sqrt{T-t},\ \tau\le t<T,
\ee
for some $M_1,\bar\rho>0$ and $\tau\in(0,T)$. 
To this end we introduce the similarity variables
\be{VraieE0def}
\phi (y,s)=(T-t)w\big(y\sqrt{T-t},t\big),\quad y=\frac{r}{\sqrt{T-t}},\ s=-\log (T-t).
\ee
By direct computation, we see that $\phi$ is a global, radially symmetric solution of
\be{VraieE0}
\phi_s=\phi_{yy}+\big(\ts\frac{n+1}{y}-\ts\frac{y}{2} \big)\phi_y-\phi+n\phi^2+y\phi\phi_y,\quad \hbox{$y\in D(s):=(0,Re^{s/2})$ and 
$s>-\log T$,}
\ee
along with $\phi_y(0,s)=0$.
Moreover, $\phi$ is globally bounded in view of \eqref{pfmonotB0b} and $w_r\le 0$. 
It then follows from interior parabolic estimates that there exists $M_1>0$ such that 
$|\phi_y|\le M_1$ for 
$s\in[s_1,\infty)$ and $0<y\le\bar y$,
with $s_1=\max(s_0+1,0)$ and $\bar y=\min(1,R/2)$.
Since $\phi_y(y,s)=(T-t)^{3/2}w_r(r,s)$, this implies \eqref{pfmonotB0}.

Setting $M_2:=\min(\bar\rho,(2M_1)^{-1})$, 
it follows from \eqref{pfmonotB0b}, \eqref{pfmonotB0} that
\be{pfmonotB1}
\begin{aligned}
w(x,t)\ge 
(T-t)^{-1}-M_1|x|(T-t)^{-3/2}&\ge  \ts\frac12 (T-t)^{-1},\\ 
\noalign{\vskip 1mm}
&\quad\hbox{for $|x|\le M_2\sqrt{T-t}$ and $\tau\le t<T$.}
\end{aligned}
\ee
Fix $x_0$ with $|x_0|<\sigma:=\min(R,\frac12 M_2\sqrt{T-\tau})$
and set $t_0=T-4M_2^{-1}|x_0|^2$.
By \eqref{pfmonotB1}, we have 
\be{pfmonotB1c}
w(x_0,t)\ge \ts\frac12 (T-t)^{-1},\quad \tau\le t\le t_0. 
\ee
To estimate $w(x_0,t)$ in the remaining time interval $[t_0,T)$, we will use a comparison argument in the cylinder 
$\tilde B\times[t_0,T)$ where $\tilde B:=\tilde B_\delta(\tilde x_0)$, $\delta:=2|x_0|$, $\tilde x_0=(x_0,0,0)$.
Set $M_3:=\frac18 M_2$. 
By \eqref{pfmonotB1}, at the initial time $t_0$ we have
\be{pfmonotB2} 
w(x,t_0)\ge\ \ts\frac12 (T-t_0)^{-1}=M_3|x_0|^{-2},\quad\hbox{for } |x|\le M_2\sqrt{T-t_0}=2|x_0|.
\ee
	Denote by $\varphi$ the first eigenfunction of the negative Dirichlet Laplacian in $\tilde B_{1}$ normalized by $\|\varphi\|_\infty=\varphi(0)=1$ and $\lambda_1>0$ the corresponding eigenvalue.
	Set 
 $\phi_{\tilde B}(\tilde x)=\varphi((\tilde x-\tilde x_0)/\delta)$ and $\lambda_\delta=\frac{\lambda_1}{\delta^2}$.
Also denote by $S_{\tilde B}$ the Dirichlet heat semigroup on $ \tilde B$.
Since $w\ge 0=S_{\tilde B}(t-t_0)\phi_{\tilde B}$ on $\partial \tilde B\times (t_0,T)$, 
it follows from the maximum principle that, for $(\tilde x,t)\in \tilde B\times (t_0,T)$, 
$$\begin{aligned}
w(\tilde x,t):=w(x,t)
&\ge S_{\tilde B}(t-t_0)M_3|x_0|^{-2}\\
&\ge M_3|x_0|^{-2} S_{\tilde B}(t-t_0)\phi_{\tilde B}
=M_3|x_0|^{-2}\exp(-\lambda_\delta(t-t_0))\phi_{\tilde B}\\
&\ge M_3|x_0|^{-2}\exp\Bigl(-\frac{\lambda}{4|x_0|^2}(T-t_0)\Bigr)\phi_{\tilde B}=M_4|x_0|^{-2}\phi_{\tilde B},
\end{aligned}$$
where $M_4=M_3\exp(-\frac{\lambda}{M_2})$.
In particular, 
\be{pfmonotB3}
w(\tilde x_0,t)=w(x_0,t)\ge M_4|x_0|^{-2},\quad t_0<t<T.
\ee
Setting $M_5=\min(M_4,\frac12)$ and combining \eqref{pfmonotB1c} and \eqref{pfmonotB3}, we obtain
$$w(x_0,t)\ge M_5\big(T-t+|x_0|^2\big)^{-1},\quad \tau\le t\le T,$$
hence \eqref{estiminfwmonot}.
\end{proof}

\section{Type I blow-up: Proof of Proposition~\ref{propA1}.} \label{SecTypeI}
\label{sectype1}

In what follows, we denote by $Z_{[a,b]}\left( f\right)$ the 
 zero number of a function $f$ in the interval $[a,b]$.
The proof is based on zero number arguments, 
following ideas from \cite{MM04,CFG,MS2}.
To this end we introduce the normalized regular steady state $W_1$ of \eqref{e1}, namely the solution of
\be{eqW1}
\begin{cases}
W_1''+\frac{n+1}{r}W_1'+nW_1^2+rW_1W_1'=0,\quad r>0,\\
W_1(0)=1,\quad W_1'(0)=0,
\end{cases}
\ee
which exists globally and satisfies $W_1>0$ and $W'_1\le 0$ (see Remark~\ref{Proof-lemA1B}). 
We also denote $W_*(r):=2/r^2$ ($r>0$), 
which is a singular solution of the ODE in \eqref{eqW1}. 
The proof of Proposition~\ref{propA1} relies on the following three lemmas.
The first one, which concerns the intersection property of regular and singular steady-states of \eqref{eqW1}, was obtained in \cite[Lemma 2.4]{MS2}.

\begin{lem}\label{lemA1B}
For $3\le n\le 9$, $W_1-W_*$ has infinitely many (all nondegenerate) zeros in $[0,\infty)$. 
\end{lem}

\begin{lem}\label{lemA0}
Assume that $u_0$ satisfies \eqref{i0}, \eqref{hyp13.1b} and $T(u_0)<\infty$.
There exist $N\in \N$, $R_0, R_1\in(0,\infty)$ and $t_0,t_1\in(0,T)$, such that 
\be{SpropMatanoMerleMonotone}
w_t(0,t)>0,\ t\in [t_0,T),\qquad Z(t):=Z_{[0,R_0]}(w_t(\cdot,t))\ \hbox{is constant on $[t_0,T)$,}
\ee
and 
\be{lemA1}
Z_{[0,R_1]}\left(w(\cdot,t)-W_*\right)\le N\quad \hbox{for all } t\in[t_1,T).
\ee
\end{lem}

\begin{lem}\label{propA1app}
Assume that $u_0$ satisfies \eqref{i0}, \eqref{hyp13.1b}, $T(u_0)<\infty$
and set $m(t)=w(0,t)$.
If
\be{SpropMatanoMerleCase1}
\limsup_{t\to T} \ (T-t)m(t)=\infty,
\ee
then there exists $t_j\to T$ such that 
\be{cvgappendix}
\frac{w(r/\sqrt{m(t_j)},t_j)}{m(t_j)}\longrightarrow W_1(r),\quad\hbox{uniformly for $r\ge 0$ bounded.}
\ee
\end{lem}

Let us temporarily admit Lemmas~\ref{lemA0} and \ref{propA1app} and give the:

\begin{proof}[Proof of Proposition~\ref{propA1}] 
Assume for contradiction that, under the assumptions of the proposition, $u$ is a type II solution of (\ref{0}).
Since $u\le nw$ by \eqref{reluw}, $w$ thus satisfies \eqref{SpropMatanoMerleCase1}. 
 Let $t_1,N$ and $R_1$ be as in Lemma~\ref{lemA0} and $(t_j)_{j\in \N}$ and $m(t_j)$ are as in Lemma~\ref{propA1app} such that $t_j\ge t_1$ for all $j$.
 From Lemma~\ref{lemA1B}, there exists 
 $R_2>0$ such that $W_1-W_*$ has at least $N+1$ nondegenerate zeros on $[0,R_2]$. 
 Therefore, using the scaling property of $W_*$ and Lemma~\ref{propA1app}, 
 we obtain 
$$
\begin{aligned}
Z_{[0,R_1]}\left(w(r,t_j)-W_*(r)\right)&=Z_{\big[0,R_1\sqrt{m(t_j)}\big]}\left( \frac{w(r/\sqrt{m(t_j)},t_j)}{m(t_j)}-\frac{W_*(r/\sqrt{m(t_j)})}{m(t_j)}\right)\\
&=Z_{\big[0,R_1\sqrt{m(t_j)}\big]}\left( \frac{w(r/\sqrt{m(t_j)},t_j)}{m(t_j)}-W_*(r)\right).
\end{aligned}
$$
Choosing $j$ large enough such that $R_1\sqrt{m(t_j)}>R_2$ and using $\frac{w(r/\sqrt{m(t_j)},t_j)}{m(t_j)}\to W_1(r)$ as $j\to \infty$, we have 
$$
Z_{[0,R_1]}\left( w(r,t_j)-W_*(r)\right)\ge N+1\quad \hbox{ for large } j.
$$
This contradicts \eqref{lemA1} and concludes the proof of Proposition~\ref{propA1}.
\end{proof}

\begin{rem}\label{Proof-lemA1B}
Lemma~\ref{lemA1B} is proved in \cite{MS2} under the assumption that $W_1$ is globally defined and positive
(which is sufficient for their needs).
The latter property is actually true, as well as $W'_1\le 0$. Indeed, denote by $\bar r$ the maximal existence time of $W_1$.
Setting $Z=nW_1+rW_1'$, a direct computation yields $Z'=-rW_1Z$, with $Z(0)=n>0$, hence $Z>0$ on $J=(0,\bar r)$. 
Therefore, $(r^nW_1)'\ge 0$, hence $W_1>0$ on~$J$.
Rewriting the ODE in \eqref{eqW1} as $(r^{n+1}W_1')'=-r^{n+1}W_1Z$, we get
 $W'_1\le 0$, hence $W_1$ is bounded on~$J$.
It follows that $\bar r=\infty$ since otherwise, by \eqref{eqW1},
$W_1'$ would be bounded on~$J$: a contradiction.
\end{rem}

The proof of Lemma~\ref{lemA0} will require two additional ingredients
for the case $\Omega=\R^n$.
The first one, as in \cite{MS2}, is a backward uniqueness property 
for linear parabolic equations in exterior domains (cf.~\cite[Theorem 1]{ESS-2003}).

\begin{prop}\label{backuniq}
Let $Q_{\rho,T}:=\left(\R^d \setminus B_\rho\right)\times (0,T]$. Assume that $u$ satisfies
$$
|u_t-\Delta u|\le A\left(|u|+|\nabla u|\right),\quad |u(x,t)|\le A\exp(A|x|^2),\quad \hbox{in } Q_{\rho,T},
$$ 
for some constant $A>0$. If $u(\cdot,T)\equiv 0$, then $u\equiv0$ in $Q_{\rho,T}$.
\end{prop}

 The second one is a decay property for positive solutions of the ODE in \eqref{eqW1}.
It will be used in the argument below, which enables us to avoid the decay assumption made on $u_0$
in the analytic proof from~\cite{MS2}. 

\begin{lem}\label{lemA0Rn1}
Let $V\in C^2(0,\infty)$ be a solution of 
\be{decayODE0}
V''+\frac{n+1}{r}V'+nV^2+rVV'=0,\quad V\ge 0,\quad r\in(0,\infty). 
\ee
Then for any $\eps>0$, there exists $C=C(\eps)>0$ such that 
\be{decayODE}
V(r)\le Cr^{-2}\quad \hbox{in}\ \ (\eps,\infty).
\ee
\end{lem}
\begin{proof}[Proof of Lemma~\ref{lemA0Rn1}]
Let $W(s)=r^{2}V(r)\ge 0$ with $s=\log r$. By direct computation,
 we have 
\be{eqODEW}
W''+(n-4+W)W'+(n-2)(W^2-2W)=0,\quad  s\in\R.
\ee
For $n\ge 4$, setting $E(s)=\frac12{W'}^2 +(n-2)\big(\frac13W^3-W^2\big)$,
we have 
$E'(s)=-(n-4+W){W'}^2\le 0$.
Consequently, for each $\sigma\in\R$, we have
$\sup_{s\ge \sigma}E(s)\le E(\sigma)$, hence 
$\sup_{s\ge \sigma}W(s)<\infty$, which yields the desired conclusion.

For $n=3$ we use a different energy functional.
Rewriting equation \eqref{eqODEW} as
$W''+W'+(W'+W)(W-2)=0$
and setting $h=W'+W$, we obtain $h'+(W-2)h=0$.
It follows that either: (i) $h\equiv 0$; (ii) $h<0$ on $\R$; or (iii) $h>0$ on $\R$.
In case (iii), setting $H=\frac12 (W-2)^2+h-2\log h$, we compute
$$H'= (W-2)W'+\frac{h'(h-2)}{h}
=(W-2)W'-(W-2)(W'+W-2)=-(W-2)^2\le 0.$$
Consequently, for each $\sigma\in\R$, we have
$\sup_{s\ge \sigma}H(s)\le H(\sigma)$, hence 
$M_\sigma:=\sup_{s\ge \sigma}h(s)<\infty$, and this remains obviously true 
in cases (i) and (ii).
Going back to the definition of $h$, we get
$\frac{d}{ds}[(W(s)-M_\sigma)e^{s}]\le 0$ on $[\sigma,\infty)$,
hence again $\sup_{s\ge \sigma}W(s)<\infty$.
\end{proof}

\begin{proof}[Proof of Lemma~\ref{lemA0}] 
$\bullet$ First note that, since $B(u)=\{0\}$ by assumption \eqref{hyp13.1b}
 and Remark~\ref{remopen}, $w$ extends by parabolic regularity to a function
$w\in C^{2,1}((\overline\Omega\setminus\{0\})\times (0,T])$ 
and the PDE in \eqref{e1}  is satisfied on $(\Omega\setminus\{0\})\times \{T\}$.
 We claim that:
\be{ClaimNotequiv}
\hbox{if $\Omega=\R^n$ and $V$ is a solution of \eqref{decayODE0}, then $w(\cdot,T)\not\equiv V$ in $\R^n\setminus\{0\}$.}
\ee
Assume for contradiction that $w(\cdot,T)\equiv V$ in $\R^n\setminus\{0\}$.
Recall the notation in \eqref{e1} and set $z(\tilde x,t):=w(\tilde x,t)-V(|\tilde x|)$
for $(\tilde x,t)\in\tilde Q=(\R^{n+2}\setminus\{0\})\times (0,T]$.
Using \eqref{reluw}, we have $(nw^2+rww_r)-(nV^2+rVV_r)=n(w+V)z+rw_rz+rVz_r=(u+nV)z+rVz_r$,
so that $z$ is a solution of 
\be{ClaimNotequiv2}
z_t-\tilde\Delta z=b(r,t)z_r+c(r,t)z,\quad (\tilde x,t)\in Q,\quad \hbox{ with $b=rV $ and $c=u+nV$}.
\ee
For each $\rho>0$, using Lemma~\ref{lemA0Rn1} and the fact that $B(u)=\{0\}$ and $w_r, u_r\le 0$, 
we see that $z,b,c\in L^\infty(Q_{\rho,T})$ where $Q_{\rho,T}:=(\R^{n+2}\setminus B_\rho)\times (0,T]$.
We may thus apply the backward uniqueness property in Proposition~\ref{backuniq} to deduce that $w\equiv V$, hence $w_t\equiv 0$, in $Q_{\rho,T}$.
Letting $\rho\to 0$, we deduce that $w_t\equiv 0$ in $\R^{n+2}\times (0,T)$: a contradiction with $T=T(u_0)<\infty$. 
This proves \eqref{ClaimNotequiv}.

\smallskip

$\bullet$  {\it Proof of \eqref{SpropMatanoMerleMonotone}.}  The function 
$U=w_t$ satisfies
\be{SMatanoMerleEqut}
U_t-U_{rr}=b(r,t)U_r+c(r,t)U,\quad \hbox{ with $b=\frac{n+1}{r}+rw $ and $c=2nw+rw_r$}.
\ee
First consider the case $\Omega=\R^n$. Then $U(\cdot,T)\not\equiv 0$.
Indeed since, as recalled at the beginning of the proof,  the PDE in  \eqref{e1} is satisfied on 
$(\R^n\setminus\{0\})\times \{T\}$,
$w_t(\cdot,T)\equiv 0$ would mean that $V(r):=w(r,T)$ 
is a solution of \eqref{decayODE0}, contradicting \eqref{ClaimNotequiv}.
Consequently, by continuity, there exist $R_0>0$ and $t_0\in(0,T)$ such 
that $U(R_0,t)\ne 0$ for all $t\in[t_0,T]$.
If $\Omega=B_R$, then we set $R_0=R$, $t_0=T/2$ and we have $U(R,t)=0$ for all $t\in[t_0,T)$ owing to the boundary conditions in \eqref{e1}.

 In either case,
the above allows to apply properties of the zero number 
on $[0,R_0]$ (cf.~\cite[Theorem 52.28 and Remark 52.29(ii)]{quittner_souplet},
  arguing on any time interval $[t_0,T_1)$ with $T_1<T$), and we deduce that
 the function $t\mapsto Z_{[0,R_0]}(U(\cdot,t))$ is (finite) integer-valued, nonincreasing on $[t_0,T)$,
  and drops at any time $t$ such that $U(\cdot,t)$ has a degenerate zero.
It follows that $Z_{[0,R_0]}(U(\cdot,t))$ is constant on $(\tilde t_0,T)$ for some $\tilde t_0\in (t_0,T)$.
Moreover, since $U_r(0,t)=0$, we have $U(0,t)=w_t(0,t)\ne 0$ for all $t$ close enough to $T$,
hence $w_t(0,t)>0$ owing to $\lim_{t\to T} w(0,t)=\infty$.

\smallskip

$\bullet$ {\it Proof of \eqref{lemA1}.}
By the computation leading to \eqref{ClaimNotequiv2},
$U(r,t):=w-W^*$ is a solution of \eqref{SMatanoMerleEqut} on 
$(0,R)\times(0,T)$ with $b=\frac{n+1}{r}+rw $ and $c=nW^*+u$. 
If~$\Omega=\R^n$, then $U(\cdot,T)\not\equiv 0$ owing to \eqref{ClaimNotequiv} hence, by continuity, there exist $R_1>0$ and $t_1\in(0,T)$ such 
that $U(R_1,t)\ne 0$ for all $t\in[t_1,T)$.
If~$\Omega=B_R$, then we set $R_1=R$, $t_1=T/2$ and we have 
$U(R_1,t)=C:=\mu-2R^{-2}$ 
for all $t\in[t_1,T)$ owing to the boundary conditions in \eqref{e1}.
On the other hand, for any fixed $\eps\in (0,T-t_1)$, since $\lim_{r\to 0} W^*(r)=\infty$, 
there exists $\eta_\eps\in(0,R_1)$ such that $U(r,t)<0$ for all $0<r\le \eta_\eps$ and $t_1\le t\le T-\eps$. 
Therefore, 
$Z_{[0,R_1]}(U(\cdot,t))=Z_{[\eta_\eps,R_1]}(U(\cdot,t))$ for all $t\in[t_1,T-\eps]$.
For each $\eps>0$, by the above, we may apply properties of the zero number (cf.~previous paragraph) 
to $U$
on $[\eta_\eps,R_1]\times(t_1,T-\eps)$, to infer that $t\mapsto Z_{[\eta_\eps,R_1]}(U(\cdot,t))$ is  finite and nonincreasing 
on $(t_1,T-\eps)$. Conclusion \eqref{lemA1} readily follows.
\end{proof}

\begin{rem} \label{remBwdU}
Thanks to Lemma~\ref{lemA0Rn1} and the argument following \eqref{SMatanoMerleEqut}, we do not need to apply backward uniqueness to the
equation \eqref{SMatanoMerleEqut} for $u_t$ which, like in \cite{MS2},
 would require the boundedness of the coefficient $rw$,
 leading to the additional assumption $u_0\le C(1+|x|)^{-1}$.
More precisely, we only need to apply backward uniqueness to the
equation for $w-V$ where $V$ is an arbitrary (radial) steady state on $(0,\infty)$,
which only requires the boundedness of the coefficient $rV$, 
and the latter is guaranteed by Lemma~\ref{lemA0Rn1}.
\end{rem}

With property \eqref{SpropMatanoMerleMonotone} at hand, 
Lemma~\ref{propA1app} can be proved by a similar argument as in \cite{CFG} 
(see also  \cite[Theorem 23.10]{quittner_souplet}).
We give the full proof for completeness and since the equation differs.

\begin{proof}[Proof of Lemma~\ref{propA1app}] 
By a time shift, we may assume that $t_0=0$ in Lemma~\ref{lemA0}.
Our assumptions imply the existence of a sequence $t_j\to T$ such that
\be{SpropMatanoMerleCFG}
\frac{w_t(0,t_j)}{ w^2(0,t_j)}\to 0
\ee
(otherwise we would have $w_t(0,t)\geq cw^2(0,t)$ as $t\to T$ for some $c>0$,
due to \eqref{SpropMatanoMerleMonotone}, which would contradict 
\eqref{SpropMatanoMerleCase1}).
Set $M_j=w(0,t_j)$. By comparing with the solution of the ODE
$\psi'=\psi^2$, $\psi(t_j)=M_j$,
we easily obtain the existence of $s^*>0$ such that
\be{SpropMatanoMerletildetj}
\tilde t_j:=t_j+s^* M_j^{-1}<T\quad\hbox{ and }\quad
w(0,t)\leq 2M_j,\quad t_j\leq t\leq \tilde t_j.
\ee
Let $\lambda_j=M_j^{-1/2}$ and define the rescaled solutions
$$v_j(y,s)=M_j^{-1}w(\lambda_jy,t_j+\lambda_j^2s)\qquad (y,s)\in D_j:=\big(0,R_0\lambda_j^{-1}\big)\times\big(-t_j\lambda_j^{-2},(T-t_j)\lambda_j^{-2}\big).$$
Then
$$\partial_s v_j-\partial_{yy} v_j-\ts\frac{n+1}{y}\partial_y v_j=nv_j^2+yv_j\partial_y v_j,\qquad (y,s)\in D_j$$
and, by \eqref{SpropMatanoMerletildetj} and \eqref{SpropMatanoMerleMonotone},
we have $0\leq v_j\leq 2$ in $\big(0,R_0\lambda_j^{-1}\big)\times(-t_j\lambda_j^{-2},s^*)$.
Moreover, $v_j(0,0)=1$ and $\partial_s v_j(0,0)\to 0$, due to \eqref{SpropMatanoMerleCFG}.
Let $D:=(0,\infty)\times(-\infty,s^*)$. By interior parabolic estimates, it follows that
(some subsequence of) $v_j$ converges in $C^{2+\alpha,1+\alpha/2}_{loc}(D)$ to a 
nonnegative solution of
$$
\begin{cases}
\partial_s v-\partial_{yy} v-\frac{n+1}{y}\partial_y v=nv^2+yv\partial_y v,\qquad (y,s)\in D,\\
v(0,0)=1,\ \partial_sv(0,0)=0.
\end{cases}
$$
Since $\partial_sv_j$ satisfies \eqref{SMatanoMerleEqut}, then we also have
\be{SMatanoMerleC1loc}
\partial_sv_j\to \partial_sv\quad\hbox{ in $C^{1,0}_{loc}(D)$.}
\ee

We shall now show that $\partial_sv(\cdot,0)\equiv 0$. If not, then
there exist $A>0$ and $\eps\in (0,s^*)$ such that
\be{SMatanoMerlevs}
\partial_sv(A,s)\neq 0,\quad |s|\leq\eps.
\ee
Since $\partial_sv(\cdot,0)$ has a degenerate zero at $r=0$, it follows
from properties of the zero number (cf.~\cite[Theorem 52.28]{quittner_souplet}) that the zero number of $\partial_sv$ on $[0,A]$
drops at $s=0$. Namely, we can fix
$-\eps<s_1<0<s_2<\eps$ such that $\partial_sv(\cdot,s_i)$ has only
simple zeros on $[0,A]$ and such that
$z_{[0,A]}\left(\partial_sv(\cdot,s_1)\right)\geq z_{[0,A]}\left(\partial_sv(\cdot,s_2)\right)+1$.
Owing to \eqref{SMatanoMerleC1loc}, we deduce that for $j$ large enough,
$z_{[0,A]}\left(\partial_sv_j(\cdot,s_1)\right)\geq z_{[0,A]}\left(\partial_sv_j(\cdot,s_2)\right)+1$,
hence
\be{SMatanoMerlecompz1}
z_{[0,A\lambda_j]}\left(w_t(\cdot,t_j+\lambda_j^2s_1)\right)
\geq z_{[0,A\lambda_j]}\left(w_t(\cdot,t_j+\lambda_j^2s_2)\right)+1.
\ee
On the other hand, \eqref{SMatanoMerlevs} implies  
$w_t(A\lambda_j,t_j+\lambda_j^2s)\neq 0$ for $|s|\leq\eps$ so that,
by the nonincreasing property of the zero number
(cf.~\cite[Remark 52.29.(ii)]{quittner_souplet})
\be{SMatanoMerlecompz2}
z_{[A\lambda_j,R_0]}\left(w_t(\cdot,t_j+\lambda_j^2s_1)\right)
\geq z_{[A\lambda_j,R_0]}\left(w_t(\cdot,t_j+\lambda_j^2s_2)\right).
\ee
We deduce, by \eqref{SMatanoMerlecompz1} and \eqref{SMatanoMerlecompz2},  that
$Z(t_j+\lambda_j^2s_1)\geq Z(t_j+\lambda_j^2s_2)+1$,
which contradicts \eqref{SpropMatanoMerleMonotone}.
It follows that $v_s(\cdot,0)\equiv 0$, hence
$v(\cdot,0)\equiv U_1$ due to $v(0,0)=1$, and the lemma follows.
\end{proof}

\section{Local convergence in similarity variables: Proof of Proposition~\ref{propA2}.}
\label{SecLocalCv}

The proof is based on modifications of the arguments in \cite{GMS} (see Section~\ref{ideas} for details). 
Set $R=\infty$ in case $\Omega=\R^n$. 
The rescaled $\phi(y,s)$ of $w$ by similarity variables, defined in
\eqref{VraieE0def}, is a global solution of \eqref{VraieE0}.
The type I assumption implies 
\be{wysM}
0<\phi(y,s)\le M,\quad y\in [0,Re^{s/2}),\ s\in (-\log T,\infty).
\ee
For any $\alpha\ge 0$, consider the initial value problem
associated to steady states of \eqref{VraieE0}, namely
\be{eqlimsta}
\begin{cases}
\varphi''+\ts\big(\frac{n+1}{y}-\frac12 y\big)\varphi'+n\varphi^2+y\varphi \varphi'-\varphi=0,\quad r>0,\\
\varphi(0)=\alpha\ge0,\quad \varphi'(0)=0,
\end{cases}
\ee
and denote by $y^*=y^*_\alpha$ the maximal existence time of its solution $\varphi=\varphi_\alpha$.
For $\alpha>0$, we define 
$$
r_\alpha:=\sup \{r\in (0,y^*_\alpha):\ \varphi_\alpha>0\ \hbox{on } (0,r)\}\quad \hbox{and }\  S:=\{\alpha>0:\ r_\alpha=\infty\}.
$$

The following properties of solutions of \eqref{eqlimsta} were proved in \cite[Proposition 3.1]{GMS}.

\begin{lem}\label{lemA3}
For all $\alpha\in S$, the limit $\Lambda_\alpha=\lim_{y\to \infty} y^2\varphi_\alpha(y)$ exists. Moreover, if $\alpha_1,\alpha_2\in S$ and $\alpha_1\neq\alpha_2$, then $\Lambda_{\alpha_1}\neq\Lambda_{\alpha_2}$.
\end{lem}

We shall also need the following additional property of \eqref{eqlimsta}.

\begin{lem}\label{lemA4}
For all $\alpha>0$, if $y^*=y^*_\alpha<\infty$, then $\lim_{y\to y^*}\varphi(y)=\infty$ or $-\infty$.
\end{lem}

\begin{proof}
We first note that $\varphi'(y)\ne 0$ for $y$ close to $y^*$.
Otherwise, we have $\varphi'(y_i)=0$ and $\varphi''(y_i)\ge 0$
for some sequence $y_i\to y^*$.
Then $y^*<\infty$ implies $|\varphi(y_i)|\to\infty$ hence, by \eqref{eqlimsta}, 
$\varphi''(y_i)=(\varphi-n\varphi^2)(y_i)<0$ for $i$ large:
a contradiction.

Consequently, $\varphi$ is monotone for $y$ close to $y^*$.
If the conclusion of the lemma fails, $\varphi$ is thus bounded on $(0,y^*)$. It then follows from \eqref{eqlimsta} that
$|\varphi''|=\big|\ts\big(\frac12 y-\frac{n+1}{y}-y\varphi\big)\varphi'+\varphi-n\varphi^2\big|
\le C(1+|\varphi'|)$ on $(y^*/2,y^*)$, hence $\varphi'$ is also bounded:
a contradiction with $y^*<\infty$.
\end{proof}

We will make use of the following 
nondegeneracy property for radial solutions of linear parabolic equations.

\begin{lem}\label{lemA5}
Let $R_0\in(0,\infty]$, $s_1<s_2$, $Q=(0,R_0)\times(s_1,s_2)$, $d\in\N^*$ and $a, b\in L^\infty(Q)$. If $V\in C^{2,1}(\bar Q)$ is a solution of 
\be{eqnnondeg}
V_t-V_{yy}-\frac{d-1}{y}V_y=a(y,s) V+b(y,s)V_y\quad \hbox{in $Q$}
\ee
such that
$V(0,s)=V_y(0,s)=0$ for all $s\in (s_1,s_2)$, then $V\equiv0.$
\end{lem}
\begin{proof}[Proof of Lemma~\ref{lemA5}]
Assume that there exists $(y_0,s_0)\in (0,R_0)\times (s_1,s_2)$ such that $V(y_0,s_0)\neq 0$. Then, there exists a sufficiently small $\eps>0$ such that $V(r_0,s)\neq 0$ for all $s\in [s_0-\eps, s_0+\eps]\subset (s_1,s_2)$. By
properties of the zero number (cf.~\cite[Theorem 52.28 and Remark 52.29(ii)]{quittner_souplet}), it follows that
 the function $s\mapsto Z_{[0,y_0]}(V(\cdot,s))\ge 0$ is (finite) integer-valued, nonincreasing, and drops at any degenerate zero.
Since $V$ has a degenerate zero at $y=0$ for each time, this is a contradiction.
\end{proof}

\begin{proof}[Proof of Proposition~\ref{propA2}] 
(i) Let $w$ be as in Proposition~\ref{propA2} and let $\phi$ be defined by \eqref{VraieE0def}.
Step 1 closely follows the idea in~\cite{GMS}, whereas Step~2 will simplify their arguments by making use of Lemma~\ref{lemA5}. 

\smallskip

{\bf Step 1.}  {\it Nonoscillation of $\phi(0,s)$ as $s\to\infty$.}
Recalling \eqref{wysM}, we claim that 
\be{limphiK}
\ell:=\lim_{s\to\infty} \phi(0,s) \hbox{ exists in } [0,M]. 
\ee

Assume for contradiction that
\be{limphiKnot}
\ell_1<\ell_2, \quad\hbox{where } \ell_1:=\liminf_{s\to\infty} \phi(0,s)\ge 0,\quad \ell_2:=\limsup_{s\to\infty} \phi(0,s)\le M.
\ee
We first claim that $J:=(\ell_1,\ell_2)\subset S$.
Assume $\alpha\in J\setminus S$. Since $r_\alpha<\infty$, 
two cases arise:

\noindent {\hskip 3mm}$\bullet$ $r_\alpha<y^*_\alpha$, hence $\varphi_\alpha(r_\alpha)=0$. 

\noindent {\hskip 3mm}$\bullet$ $r_\alpha=y^*_\alpha<\infty$, 
hence $\varphi_\alpha(r)\to \pm\infty$ as $r\to r_\alpha^-$ by Lemma~\ref{lemA4}.

\noindent Pick $s_1>-\log T$ such that $Re^{s_1/2}>r_\alpha$.
Using \eqref{wysM} and letting $\rho_\alpha=r_\alpha$ in the first case and $\rho_\alpha=r_\alpha-\eta$ with $\eta>0$ small
in the second case, we have $\phi(\rho_\alpha,s)-\varphi_\alpha(\rho_\alpha)\ne 0$ for all $s\ge s_1$.
By properties of the zero number, we deduce the existence of $N_1\in \N$ such that 
$$
Z_{[0,\rho_\alpha)}\left(\phi(\cdot,s)-\varphi_\alpha\right)\le N_1,\quad \hbox{ for all } s>s_1,
$$
and $s\mapsto Z_{[0,\rho_\alpha)}\left(\phi(\cdot,s)-\varphi_\alpha(\cdot)\right)$ is nonincreasing in $(s_1,\infty)$ and drops  at any degenerate zero. But $\phi(0,s)-\varphi_\alpha(0)$ changes sign infinitely many times, so that $\phi(\cdot,s)-\varphi_\alpha(\cdot)$ has infinitely many degenerate zeros in $[0,r_\alpha)$.
This is a contradiction, which proves $J\subset S$.

Now, by the assumption $B(u_0)=\{0\}$ and parabolic regularity, 
the limit $w(r,T):=\lim_{t\to T} w(r,t)$ exists for all $r\in(0,\infty)$.
For any $\alpha\in J$, set
$w_\alpha(r,t)=(T-t)^{-1}\varphi_\alpha\big((T-t)^{-1/2}r\big)$. 
Since $J\subset S$, it follows from Lemma~\ref{lemA3} that
\be{lamnda1}
w_\alpha(r,T)=\lim_{t\to T} w_\alpha(r,t)=\Lambda_\alpha r^{-2},\quad r>0.
\ee
Fix any $r_0\in(r_1,r_2)$.
Owing to the injectivity property of $\Lambda_\alpha$ in Lemma~\ref{lemA3}, we may choose $\alpha\in J$ such that 
$\Lambda_\alpha r_0^{-2}\ne w(r_0,T)$. Consequently, 
by continuity, there exists $t_0\in(0,T)$ such that 
$w(r_0,t)\neq w_\alpha(r_0,t)$ in $[t_0,T]$.
Also, the function $w-w_\alpha$ is a solution of \eqref{SMatanoMerleEqut} on 
$(0,r_0)\times(t_0,T)$ with $b=\frac{n+1}{r}+rw_\alpha$ 
and $c=u+nw_\alpha$. 
By properties of the zero number, we deduce that $Z_{[0,r_0]}\left(w(\cdot,t)- w_\alpha(\cdot,t)\right)$ is finite and nonincreasing on $(t_0,T)$ and drops at any degenerate zero. 
But, by hypothesis \eqref{limphiKnot}, the function $w(0,t)-w_\alpha(0,t)=(T-t)^{-1}\left(\phi(0,-\log(T-t))-\alpha \right)$ changes sign infinitely many times on $(t_0,T)$. This is a contradiction and \eqref{limphiK} is proved.

\smallskip

{\bf Step 2.} {\it $\omega$-limit and nondegeneracy argument.}
Let us define the $\omega$-limit set: 
$$
\omega (\phi)=\big\{ \Phi: \phi(\cdot,s_j)\to \Phi  \ \hbox{in $C^1_{loc}([0,\infty))$ for some $s_j\to \infty$  as } j\to \infty\big\}. 
$$
By \eqref{wysM} and parabolic regularity, $\omega (\phi)$ is nonempty.
Pick any $\Phi\in \omega (\phi)$,
take any sequence $s_j\to \infty$ such that $\phi(\cdot,s_j)\to \Phi$ in $L^\infty_{loc}([0,\infty))$,
and set $\phi_j:=\phi(y,s+s_j)$.
It follows from \eqref{VraieE0}, \eqref{wysM} and 
 parabolic estimates that
the sequence $(\phi_j)_j$ is  
 bounded in $W^{2,1;q}_{loc}(Q)$ for any $1<q<\infty$, where $Q=[0,\infty)\times\R$. 
Consequently, there exists a subsequence (still denoted $s_j$) and  a function $Z$
such that $\phi(\cdot,\cdot+s_j)\to Z$ weakly in $W^{2,1;q}_{loc}(Q)$ and strongly in $C^{1+\alpha,\alpha/2}_{loc}(Q)$ for any $\alpha\in(0,1)$.
 Therefore, $Z$ is a  strong solution of
\be{limeqn}
Z_s-Z_{yy}-\frac{n+1}{y}Z_y=n Z^2-\frac{1}{2}y Z_y-Z+y ZZ_y\quad \text{in }Q,
\ee
with $Z_y(0,s)=0$ for all $s\in\R$. Moreover, by \eqref{limphiK}, we have $Z(0,s)=\ell$ for all $s\in\R$. 

Set $V=Z_s$. By parabolic regularity we have $V\in C^{2,1}(Q)$ 
and we see that $V$ satisfies \eqref{eqnnondeg} with $d=n+2$, $a=2nZ-1+yZ_y$ and $b=y(Z-\frac12)$,  
along with $V(0,s)=V_y(0,s)=0$ for all $s\in\R$. 
By Lemma~\ref{lemA5}, applied for each $R_0>0$, 
we deduce that $V\equiv 0$.
Consequently $Z$ is a steady-state of  \eqref{limeqn}, hence $Z=\varphi_\ell$.
We have shown that $\omega (\phi)=\{\varphi_\ell\}$,
which proves \eqref{cvpropA2w} with $\Psi=\varphi_\ell$, 
hence \eqref{cvpropA2} in view of \eqref{reluw}.

\smallskip

(ii) We now have $w_r\le 0$ by Proposition~\ref{locprop}(ii),
hence in particular $\Psi'\le 0$.
Moreover, by \eqref{pfmonotB0b}, we have
\be{Psinot0}
\Psi(0)\ge 1/n.
\ee
Next, by \cite[Theorem~1.3 and Remark~(vi)~p.670]{Soup-Win}, we know that
$$w(x,t)\le C_1(T-t+|x|^2)^{-1}\quad\hbox{ in $B_R\times(T/2,T)$}$$
 for some $C_1>0$, hence
$(T-t)w(y\sqrt{T-t},t)\le C_1(1+|y|^2)^{-1}$ for each $y>0$ and $t\in(T-R^2|y|^{-2},T)$. 
Applying this with $y=\sqrt{nC_1}$ and letting $t\to T$ in  \eqref{cvpropA2w}
guarantees that \be{Psinot1n}
\Psi\not\equiv 1/n.
\ee
Finally, if $\Psi'(y)=0$ for some $y>0$, then
$\Psi''(y)=0$, owing to $\Psi'\le 0$, hence $\Psi(y)=0$ or $1/n$ by \eqref{eqnPsi}, so that $\Psi\equiv 0$ or $1/n$
by local uniqueness for \eqref{eqnPsi}: a contradiction with \eqref{Psinot0} or \eqref{Psinot1n}.
This concludes the proof. 
\end{proof}

{\bf Declarations:}  The authors report there are no competing interests to declare.
This manuscript has no associated data.

\section{Appendix. Proof of Proposition~\ref{locprop}}

(i) $\bullet$ {\it Local existence-uniqueness and regularity for $\Omega=B_R$.}
By \cite[Theorem~1 and Example~1]{BN93}
(actually also valid in nonradial setting for any $0\le u_0\in L^q(B_R)$ with $q\in(n/2,\infty)$; see also \cite[Theorem~2(ii)]{BN94}),
there exists a unique maximal weak solution, with $u\in C([0,T);L^q(\Omega))$ for all $q\in(n/2,\infty)$.
The constructed solution satisfies $u\ge 0$ and $u\in L^\infty_{loc}(0,T;L^\infty(\Omega))$.
The latter property and a bootstrap argument using parabolic regularity
guarantee the classical regularity of the solution for $t>0$.
Moreover, in view of the radial symmetry of $u_0$, the radial symmetry of $u$ follows from uniqueness.
We claim that 
\be{bddinitial}
u, w\in L^\infty_{loc}([0,T);L^\infty(\Omega)).
\ee
Since $u\in C([0,T);L^q(\Omega))$ for all finite $q$, the second equation and elliptic regularity imply $\nabla v\in L^\infty_{loc}([0,T);L^\infty(\Omega))$,
so that $u$ solves $u_t-\Delta u=a(x,t)u+b(x,t)\cdot\nabla u$ with $a\in L^\infty_{loc}([0,T);L^q(\Omega))$
and $b\in L^\infty_{loc}([0,T);L^\infty(\Omega))$.
The linear problem $\bar u_t-\Delta \bar u=a(x,t)\bar u+b(x,t)\cdot\nabla \bar u$ with $\bar u_\nu=0$ and $\bar u(0)=u_0$
admits a unique strong solution such that $\bar u-S(t)u_0\in C([0,T);L^\infty(\Omega))$, where $S(t)$ is the Neumann semigroup.
Also,  since $v=0$ on $\partial B_R$ and $v\ge 0$ by the maximum principle, we have 
$v_\nu\le 0$ hence $u_\nu=uv_\nu\le 0$ on $\partial B_R$.
Since $a\in L^\infty_{loc}([0,T);L^q(\Omega))$ ($q>n/2$) and $u,\bar u\in C([0,T);L^2(\Omega))$
we may compare $u$ and $\bar u$ by the (Stampacchia) maximum principle, to get $u\le \bar u$ in $\overline\Omega\times(0,T)$, 
hence \eqref{bddinitial} for $u$.
Property \eqref{bddinitial} for $w$ is then a direct consequence of that for $u$ and of definition \eqref{w}.

\smallskip

$\bullet$ {\it Local existence-uniqueness and regularity for $\Omega=\R^n$.}
Since $u_0$ is radially symmetric and we look for radial solutions, the second equation is equivalent to
$-v_r(r,t) = \int_0^r (s/r)^{n-1} u(s,t)ds =: \mathcal{K}(u)$ 
where $r=|x|$, so that the problem becomes equivalent 
 to solving $u_t-\Delta u = u^2+\mathcal{K}(u)u_r$ with $u(\cdot,0)=u_0\ge 0$.
For  $u_0\in BC(\R^n)$, this is done in \cite{Win23},
whose existence part relies on an approximation by compactly supported initial data.
This approach can be extended to $u_0\in L^\infty(\R^n)$ by using the existence results in \cite{biler1995b}
which apply in particular for $u_0\in L^1\cap L^\infty(\R^n)$, hence to compactly supported $u_0\in L^\infty(\R^n)$.
Alternatively, this can be done directly by a 
fixed point on a ball of the space 
$X=L^\infty(0,T_0;L^\infty(\R^n)) \cap L^\infty_{loc}((0,T_0];BC^1(\R^n))$ 
with $T_0>0$ small, 
normed by $\sup_{t\in(0,T_0)} \big[\|u(t)\|_\infty + t^{1/2} \|u_r(t)\|_\infty\big]$
(see,~e.g.,~\cite[Theorem 15.2, Example 51.30 and Proposition~51.40]{quittner_souplet} for similar arguments).

\smallskip

(ii) To prove the radial monotonicity, we start with smooth initial data, and then treat the general case by an approximation argument.

\smallskip

{\bf Case 1:} Let $u_0\in C^\infty_0(\Omega)$ with $u_0$ radially symmetric and $u_{0,r}\le 0$.
In Case~1, $\tau$ will denote any time in $(0,T)$, we will set $Q_\tau=\Omega\times(0,\tau)$ and 
$C$ will stand for generic positive constants possibly depending on $u_0$ and $\tau$.
\smallskip

$\bullet$ We claim that 
$u_r, w_r, w_{rr}\in C(\overline\Omega\times[0,T))$.
Since 
$w_t-\tilde\Delta w=f:=u w$ in $\tilde Q_T:=\tilde\Omega\times (0,T)$ by \eqref{reluw}, \eqref{highereq},
we see from \eqref{bddinitial} that $f\in L^\infty(\tilde Q_\tau)$. 
In case $\Omega=B_R$,
$w_0$ is $C^\infty$ and satisfies the first and second order compatibility 
conditions $w=\mu=R^{-n}\|u_0\|_1$ 
and $\tilde\Delta w+f=0$ on $\partial \tilde\Omega\times\{0\}$.
It follows from interior or interior-boundary parabolic $L^q$ and Schauder regularity that
$w\in W^{2,1;q}_{loc}(\overline Q_\tau)\subset C^{1+\alpha,\alpha/2}_{loc}(\overline Q_\tau)$ for some $\alpha\in(0,1)$ and some large $q\in(1,\infty)$,
hence $u=rw_r+nw\in C^{\alpha,\alpha/2}_{loc}(\overline Q_\tau)$,
and then that $w\in C^{2+\alpha,1+\alpha/2}_{loc}(\overline Q_\tau)$.
Using $u_r=rw_{rr}+(n+1)w_r$, owing to \eqref{reluw}, this implies the claim. 

\smallskip

$\bullet$ We claim that $w_r\le 0$.
The function $z:=w_r\in C(\overline\Omega\times[0,T))$ satisfies 
\be{eqnzMP}
z_t-z_{rr}-\ts\frac{d}{r}z_r+\frac{d}{r^2}z=az+bz_r,
\ee
where $d=n+1$, $a=(2n+1)w+r w_r$, $b= rw$ and $z(0,t)=0$, $z(\cdot,0)\le 0$.
By \eqref{reluw}, \eqref{bddinitial},  we have $|rw_r|=|nw-u|\le C$ in $Q_\tau$, hence 
$a\in L^\infty(Q_\tau)$ and, in case $\Omega=\R^n$,
$z\to 0$ as $r\to \infty$ uniformly for $t\in(0,\tau)$. 
The claim then follows from the maximum principle
(see, e.g.,~\cite[Proposition~52.4 and Remark~52.11(a)]{quittner_souplet}).

\smallskip

$\bullet$ Assume $\Omega=\R^n$. We claim that 
\be{decayur}
\hbox{$u_r(r,t)\to 0$ as $r\to \infty$, uniformly for $t\in(0,\tau]$.}
\ee
To this end we first note that $w_0(x)\le C(1+|x|^2)^{-n/2}$ by \eqref{w}.
On the other hand, denoting by $\tilde S(t)$ the heat semigroup in $\R^{n+2}$,
we have $\tilde S(t)(1+|x|^2)^{-n/2}\le (1+|x|^2)^{-n/2}$ for all $t>0$ owing to $-\tilde\Delta (1+|x|^2)^{-n/2}\ge 0$. 
It then follows from \eqref{highereq} and \eqref{bddinitial} that
$w(t)\le Ce^{Ct}\tilde S(t)(1+|x|^2)^{-n/2}\le C(1+|x|)^{-n}$ in $Q_\tau$, 
and $|rw_r|\le C$ guarantees $|w_r|\le C(1+|x|)^{-n-1}$ in $\bar Q_\tau$.
On the other hand, $rw_{0,rr}=u_{0,r}-(n+1)w_{0,r}$ implies $|w_{0,rr}|\le C(1+|x|)^{-n-2}$.
For any $h$ with $|h|=1$ and $R>1$, since $f=uw\le CR^{-2n}$ in $\tilde Q_R:=B_R(2Rh)\times(0,\tau)$,
using parabolic regularity on \eqref{highereq} similarly as above, we get $|w_{rr}(r,t)|\le C(1+r)^{-n-2}$ in $Q_\tau$.
Since $u_r=rw_{rr}+(n+1)w_r$, property \eqref{decayur} follows.
\smallskip

$\bullet$  We claim that $u_r\le 0$.
Differentiating in $r$ the first PDE in  \eqref{0}, rewritten as $u_t-u_{rr}-\frac{n-1}{r}u_r=u^2-u_rv_r$,
it follows that $z:=u_r\in C(\overline\Omega\times[0,T))$ solves \eqref{eqnzMP} with $d=n-1$, 
$a=2u-v_{rr}$, $b=-v_r$. 
Since $-r^{n-1}v_r=\int_0^r s^{n-1} u(s,t)\, ds$, we deduce from \eqref{bddinitial} that $a\in L^\infty(Q_\tau)$.
Also, as noted in part (i) 
above, we have $u_r(R,t)\le 0$ in case $\Omega=B_R$. 
Since $u_r(0,t)=0$ and $(u_0)_r\le 0$, 
using also \eqref{decayur} if $\Omega=\R^n$,
we may apply the maximum principle (cf.~supra) to deduce the claim.

\smallskip

{\bf Case 2:} Let $u_0$ satisfy \eqref{i0}.
\smallskip

$\bullet$  We may find a sequence of radially symmetric functions $u_{0,j}\in C^\infty_0(\Omega)$ with $\partial_r u_{0,j}\le 0$, such
that $\|w_{0,j}\|_\infty\le\|u_{0,j}\|_\infty\le \|u_0\|_\infty+1$ and $u_{0,j}\to u_0$,
$w_{0,j}\to w_0$ a.e.~and in $L^q(B_R)$ or in $L^q_{loc}(\R^n)$.
Denote by $(u_j,v_j)$ (resp.,~$w_j$) the corresponding solution of \eqref{0} (resp.,~\eqref{e1}),
with $\mu$ replaced by $\mu_j:=R^{-n}\|u_{0,j}\|_1$ if $\Omega=B_R$. 
Set $\hat Q=\bar B_R\times(0,\tau]$ or $\R^n\times(0,\tau]$.

\smallskip
$\bullet$  We claim that
$w_j$ converges in $W^{2,1;q}_{loc}(\hat Q)$
to a strong, hence classical solution $W$ of \eqref{e1} in $\hat Q$ and that,
consequently, $u_j=rw_{j,r}+nw_j\to U:=rW_r+nW$ in $C_{loc}(\hat Q)$,
so that $U_r\le 0$.
To this end, we first note that $u_{j,r}, w_{j,r}\le 0$ by Case 1.
Since also $v_{j,r}\le 0$, the function
$m_j(t):=u_j(0,t)=\|u_j(t)\|_\infty$ satisfies $m_j'\le m_j^2$ by the first PDE in  \eqref{0},
and $m_j(0)\le \|u_0\|_\infty+1$.
Therefore, there exists $t_0\in (0,T)$ such that $|u_j|\le C$ in $Q_{t_0}$, hence $|w_j|\le C$,
as well as $|rw_{j,r}|=|u_j-nw_j|\le C$ in $Q_{t_0}$
(here and below $C>0$ denotes generic constants independent of $j$). 
The claimed convergence of $w_j$ then follows from parabolic estimates applied to \eqref{highereq}.

\smallskip

$\bullet$  We now identify the initial data of $W$. Namely, we claim that $W(t)\to w_0$ in 
$L^2(B_R)$ or $L^2_{loc}(\R^n)$ as $t\to 0$.
To this end we set $\tilde \mu_j=\mu_j$, $\tilde \mu=\mu$ if $\Omega=B_R$  
and $\tilde \mu_j=0$, $\tilde \mu=0$ if $\Omega=\R^n$,  
and observe that $\tilde w_j=w_j-\tilde \mu_j$ is a solution of 
$\tilde w_j=\tilde S(t)\tilde w_{j,0}+\int_0^t \tilde S(t-s)(nw_j+rw_{j,r})w_j(s)\, ds$ 
for $t\in (0,t_0)$,
where $\tilde S(t)$ is the heat semigroup in $\R^{n+2}$ if $\Omega=\R^n$, or in $\tilde B_R$ with Dirichlet conditions if $\Omega=B_R$.
Since $\|\tilde S(t)\phi\|_\infty\le \|\phi\|_\infty$ for any $\phi\in L^\infty$,
it follows from the above uniform estimates of $w_j$ and $rw_{j,r}$ that
$|\tilde w_j(t)-\tilde S(t)\tilde w_{j,0} |\le Ct$ in $Q_{t_0}$.
For each $t\in(0,t_0)$, we have $\tilde S(t)\tilde w_{j,0}\to \tilde S(t)(w_0-\tilde \mu)$ pointwise in $\tilde\Omega$ as $j\to\infty$ (by dominated convergence).
Since also $\tilde \mu_j\to \tilde \mu$, we obtain
$|W(t)-\tilde \mu-\tilde S(t)(w_0-\tilde \mu)|\le Ct$, hence
$|W(t)-S(t)w_0+\tilde \mu(S(t)1-1)|\le Ct$ in $Q_{t_0}$.
 Since  $\lim_{t\to 0}\|\tilde S(t)1-1\|_2=0$ if $\Omega=B_R$, the claim follows.

\smallskip

$\bullet$ We claim that $w=W$ in $Q_{t_0}$.
To this end, we set $\phi=w-W$ which, by direct calculation using~\eqref{reluw}, 
solves
$\phi_t - \tilde\Delta \phi=a\phi+bx\cdot\tilde\nabla \phi$, 
where $a=u+nW$ and  $b=W$ 
satisfy $a,b\in L^\infty(Q_{t_0})$.
If $\Omega=B_R$, since $\phi=0$ on $\partial\Omega$ and $\phi(t)\to 0$ in $L^2(B_R)$ as $t\to 0$, it follows from the maximum principle that $\phi\equiv 0$.
Next assume $\Omega=\R^n$. We set $\psi(x)=\log(1+|x|^2)$, which satisfies $|\tilde\nabla\psi|=2|x|/(1+|x|^2)$ 
and $|\tilde\Delta\psi|\le c(n)$,
and we penalize $\phi$ by considering $\phi_\eps:=\phi e^{-At}-\eps\theta$ with $A=\|a\|_{L^\infty(Q_{t_0})}$, $\theta=\psi+Kt$ 
 and $K=c(n)+2\|b\|_{L^\infty(Q_{t_0})}$.
Then we have
$B:=-\theta_t+ \tilde\Delta\theta+bx\cdot\tilde\nabla\theta\le-K+c(n)+2\|b\|_{L^\infty(Q_{t_0})}=0$, hence 
$$\partial_t \phi_\eps- \tilde\Delta \phi_\eps=(a-A)\phi  e^{-At} 
+bx\cdot\nabla \phi_\eps+\eps B\le (a-A)\phi_\eps+bx\cdot\nabla \phi_\eps
\quad\hbox{in $\R^{n+2}\times(0,t_0)$.}$$
For each $\eps>0$, since $\phi\in L^\infty(Q_{t_0})$, there exists $R_\eps>0$ such that 
$\phi_\eps<0$ in $(\R^{n+2}\setminus\partial B_{R_\eps})\times(0,t_0)$. 
Since $\phi_\eps(t)\le \phi(t) e^{-At}\to 0$ in $L^2(B_{R_\eps})$ 
as $t\to 0$, it follows from the maximum principle that $\phi_\eps\le 0$ in 
$B_{R_\eps}\times(0,t_0)$ hence in $\R^{n+2}\times(0,t_0)$. 
Exchanging the roles of $w$ and $W$, the claim follows.

\smallskip

$\bullet$ By the above, we have $u=rw_r+nw=rW_r+nW=U$, hence $u_r\le 0$ in $Q_{t_0}$.
Letting $T_1=\sup\{\tau\in (0,T); u_r\le 0 \hbox{ in $Q_\tau$}\}$, we necessarily have $T_1=T$ since, otherwise,
$u_r(\cdot,T_1)\le 0$ and the above argument applied with $u(\cdot,T_1)$ as new initial data would lead to a contradiction.
This concludes the proof of assertion (ii). 

\smallskip

(iii) Property \eqref{mass} follows by integrating in space the first equation if $\Omega=B_R$,
or the corresponding variation of constants formula if $\Omega=\R^n$ (then also using the conservation property of the Gaussian semigroup).
\qed

}


\begin{thebibliography}{99}

\footnotesize
\parskip=1pt
 \itemsep=1pt


\bibitem{BZ}
\sc  X. Bai, M. Zhou,
\it Exact blow-up profiles for the parabolic-elliptic Keller-Segel system in dimensions $N\ge 3$,
     \rm Math. Ann. 392 (2025), 313-337.
 
  \bibitem{biler1995}
    \sc P. Biler,
  \it  Existence and nonexistence of solutions for a model of gravitational interaction of particles III,
       \rm Colloq. Math.  68 (1995), 229-239.
       
  \bibitem{biler1995b}
    \sc P. Biler,
  \it  The Cauchy problem and self-similar solutions for a nonlinear parabolic equation, 
       \rm Stud. Math. 114 (1995), 181-205.
  
  \bibitem{BilerBook}
    \sc P. Biler,
     \it Singularities of solutions to chemotaxis systems,
            \rm De Gruyter Series in Mathematics and Life Sciences 6, Berlin: De Gruyter xxiv, 205p. (2020).

  \bibitem{BHN}
   \sc  P. Biler, D. Hilhorst, T. Nadzieja,
  \it  Existence and nonexistence of solutions for a model of gravitational interaction of particles II,
      \rm   Colloq. Math.  67 (1994), 297-308.
      
       \bibitem{BN93}
   \sc  P. Biler, T. Nadzieja, 
     \it  A class of nonlocal parabolic problems occurring in statistical mechanics, 
       \rm Colloq. Math. 66 (1993), 131-145.
       
       \bibitem{BN94}
   \sc  P. Biler, T. Nadzieja, 
     \it   Existence and nonexistence of solutions for a model of gravitational interaction of particles, I,
       \rm Colloq. Math. 66 (1994), 319-34.

       \bibitem{BKLN}
          \sc  P. Biler, G. Karch, Ph. Lauren\c cot, T. Nadzieja, 
    \it The $8\pi$-problem for radially symmetric solutions of a chemotaxis model in a disc,
       \rm Topol. Methods Nonlinear Anal. 27 (2006), 133-147.
      
        \bibitem{BCM}
   \sc   A. Blanchet, J. Carrillo, N. Masmoudi,
    \it Infinite time aggregation for the critical Patlak-Keller-Segel model in $\R^2$,
       \rm Commun. Pure Appl. Math. 61 (2008), 1449-1481.
      
  \bibitem{BDP}
   \sc   A. Blanchet, J. Dolbeault, B. Perthame,
  \it Two-dimensional Keller-Segel model: optimal critical mass and qualitative properties of the solutions,
       \rm  Electron. J. Differential Equations {\bf 44} (2006).

  \bibitem{BCKSV}
    \sc M.P. Brenner, P. Constantin, L.P. Kadanoff, A. Schenkel, S.C. Venkataramani, 
    \it Diffusion, attraction and collapse,
        \rm Nonlinearity 12 (1999), 1071-1098.
     
            \bibitem{BDPM23}  
           \sc F. Buseghin, J. D\'avila, M. del Pino, M. Musso,
         \it Existence of finite time blow-up in Keller-Segel system,
     \rm arXiv:2312.01475.

       \bibitem{BDPM25}   
           \sc F. Buseghin, J. D\'avila, M. del Pino, M. Musso,
         \it Finite-time blow-up for the three dimensional axially symmetric Keller-Segel system,
     \rm arXiv:2508.08103.
 
  \bibitem{CFG}  
  \sc       X. Chen, M. Fila, J.-S. Guo, 
      \it  Boundedness of global solutions of a supercritical parabolic equation,
       \rm   Nonlinear Anal. 68 (2008), 621-628.
     
\bibitem{Collot-2022}
  \sc C. Collot, T. Ghoul, N. Masmoudi,  V.-T. Nguyen,
  \it Refined description and stability for singular solutions of the 2d Keller-Segel system,
      \rm Comm. Pure Appl. Math. 75 (2022), 1419-1516.


\bibitem{Collot-2023}
  \sc C. Collot,  T. Ghoul, N. Masmoudi, V.-T. Nguyen, 
  \it Collapsing-ring blowup solutions for the Keller-Segel system in three dimensions and higher,
     \rm  J. Funct. Anal. 285 (2023), 110065.
     
     \bibitem{Collot-2024}
  \sc C. Collot, T. Ghoul, N. Masmoudi, V.-T. Nguyen, 
  \it Singularity formed by the collision of two collapsing solitons in interaction for the 2{D} Keller-Segel system,
     \rm  arXiv:2409.05363.


\bibitem{Collot-Zhang}
  \sc C. Collot, K. Zhang, 
  \it On the stability of Type I self-similar blowups for the Keller-Segel system in three dimensions and higher,
      \rm  arXiv:2406.11358.
      
\bibitem{CPZ}
   \sc L.  Corrias, B. Perthame, H. Zaag,
  \it Global solutions of some chemotaxis and angiogenesis systems in high space dimensions,
     \rm    Milan J. Math 72 (2004), 1-29.
 
\bibitem{Davila-2020}
   \sc J. D\'avila, M. del Pino, J. Dolbeault, M. Musso, J. Wei, 
     \it Existence and stability of infinite time blow-up in the Patlak-Keller-Segel system,
           \rm Arch. Ration. Mech.  Anal. 248 (2024),  61. 
 
\bibitem{Dolbeault-2004}
   \sc J. Dolbeault, B. Perthame, 
     \it Optimal critical mass in the two-dimensional Keller-Segel model in $\mathbb{R}^2$,
           \rm C. R. Math. Acad. Sci. Paris 339 (2004) 611-616.

\bibitem{ESS-2003}
  \sc L. Escauriaza, G. Seregin, V. $\check{S}$ver\'ak, 
	\it Backward uniqueness for parabolic equations,
	 \rm Arch. Ration. Mech. Anal. 169 (2003), 147-157.

\bibitem{FML}
  \sc A. Friedman, B. McLeod,
  \it Blow-up of solutions of semilinear heat equations,
  \rm Indiana Univ.~Math.~J.  34 (1985), 425-447.
  
\bibitem{Ghoul-2018}
   \sc T. Ghoul, N. Masmoudi, 
     \it Minimal mass blowup solutions for the Patlak-Keller-Segel equation,
           \rm Commun. Pure Appl. Math. 71 (2018), 1957-2015.

   \bibitem{GK}
  \sc    Y. Giga, R.V. Kohn,
 \it    Asymptotically self-similar blow-up of semilinear heat equations, 
    \rm  Comm. Pure Appl. Math. 38 (1985), 297-319.

   \bibitem{GMS}
  \sc    Y. Giga, N. Mizoguchi, T. Senba,
 \it    Asymptotic behavior of type I blow-up solutions to a parabolic-elliptic system of drift-diffusion type.
    \rm  Arch. Ration. Mech. Anal. 201 (2011), 549-573.

\bibitem{GS-2024}
  \sc  I. Glogi\'c, B. Sch\"orkhuber, 
\it Stable singularity formation for the Keller-Segel system in three dimensions,
    \rm  Arch. Ration. Mech. Anal. 248 (2024), 4.

 \bibitem{GP}
  \sc   I.A. Guerra, M.A. Peletier,
   \it    Self-similar blow-up for a diffusion-attraction problem,
     \rm  Nonlinearity  17 (2004),~2137-2162.

 \bibitem{HMV}
  \sc   M.A. Herrero, M. Medina, J.J.L. Vel\'azquez,
    \it    Self-similar blow-up for a reaction-diffusion system,
  \rm Journal of Computational and Applied Mathematics 97 (1998), 99-119.
  
   \bibitem{HV}
  \sc   M.A. Herrero, J.J.L. Vel\'azquez,
    \it    Singularity patterns in a chemotaxis model,
  \rm Math. Ann.  306 (1996), 583-623.
  
\bibitem{Horst1}
  \sc D. Horstmann,
  \it From 1970 until present: the Keller-Segel model in chemotaxis and its consequences, I,
    \rm Jahresber. DMV 105 (2003), 103-165.

\bibitem{Horst2}
  \sc D. Horstmann,
  \it From 1970 until present: the Keller-Segel model in chemotaxis and its consequences, II,
    \rm Jahresber. DMV 106 (2004), 51-69.

  \bibitem{HNS}
    \sc T.Y. Hou, V.T. Nguyen, P. Song,
  \it  Axisymmetric type II blowup solutions to the three dimensional Keller-Segel system,
   \rm arXiv:2502.19775.

\bibitem{jaeger_luckhaus}
  \sc W. J\"ager, S. Luckhaus, 
  \it On explosions of solutions to a system of partial differential equations modelling chemotaxis.
  \rm Trans.~Am.~Math.~Soc. 329 (1992), 819-824.
  
  \bibitem{KaSou}
   \sc  N. Kavallaris, Ph. Souplet, 
  \it Grow-up rate and refined asymptotics for a two-dimensional Patlak-Keller-Segel model in a disk,
  \rm SIAM J. Math. Anal. 40 (2009), 852-1881.
  
\bibitem{keller_segel}
  \sc E.F. Keller, L.A. Segel, 
  \it  Initiation of slime mold aggregation viewed as an instability,
  \rm J.~Theoret.~Biol. {\bf 26} 399-415 (1970).
  
  \bibitem{LZ25}
  \sc   Z. Li, T. Zhou,
  \it Nonradial stability of self-similar blowup to Keller-Segel equation in three dimensions,
   \rm arXiv:2501.07073.

  \bibitem{LSU}
  \sc    O.A. Ladyzenskaja, V.A. Solonnikov, N.N. Ural'ceva, 
    \it Linear and quasi-linear equations of parabolic type, 
       \rm Amer. Math. Soc., Transl. Math. Monographs, Providence, RI, 1968.
 
  \bibitem{Lieb}
  \sc   G.M. Lieberman, 
    \it Second order parabolic differential equations, 
       \rm World Scientific, Singapore, 1996.
  
  \bibitem{MM04}
  \sc H. Matano, F. Merle, 
  \it On nonexistence of type II blowup for a supercritical nonlinear heat equation, 
    \rm   Comm. Pure Appl. Math. 57 (2004), 1494-1541.
    
      \bibitem{MM09}
  \sc H. Matano, F. Merle, 
  \it  Classification of type I and type II behaviors for a supercritical nonlinear heat equation,   
      \rm   J. Funct. Anal. 256 (2009), 992-1064.
  
    \bibitem{Miz22}
  \sc N. Mizoguchi, 
  \it Refined asymptotic behavior of blowup solutions to a simplified chemotaxis system,
      \rm Commun. Pure Appl. Math. 75 (2022), 1870-1886.

  \bibitem{MS}
  \sc N. Mizoguchi, T. Senba,
  \it Type II blow-up solutions to a parabolic-elliptic system,
  \rm Adv.~Math.~Sci.~Appl. 17 (2007), 505-545.
 
   \bibitem{MS2}
  \sc N. Mizoguchi, T. Senba,
  \it A sufficient condition for type I blow-up in a parabolic-elliptic system,
    \rm J. Differential Equations 250 (2011), 182-203.
  
\bibitem{nagai1995}
    \sc T. Nagai,
  \it Blow-up of radially symmetric solutions to a chemotaxis system,
       \rm Adv.~Math.~Sci.~Appl. 5 (1995), 581-601.
 
 \bibitem{Naito-2008}
    \sc Y. Naito, T. Suzuki, 
     \it  Self-similarity in chemotaxis systems,
            \rm Colloq. Math. 111 (2008), 11-34.

\bibitem{Naito-Senba-2012}
    \sc Y. Naito, T. Senba, 
      \it Blow-up behavior of solutions to a parabolic-elliptic system on higher dimensional domains,
             \rm Discrete Contin. Dyn. Syst. 32 (2012), 3691-3713.

\bibitem{Nguyen-Zaag-2023}
    \sc V.T. Nguyen, N. Nouaili,  H. Zaag, 
  \it Construction of type I-Log blowup for the Keller-Segel system in dimensions $3$ and $4$,
         \rm arXiv:2309.13932.

\bibitem{NWZ}
    \sc V.T. Nguyen, Z.-A. Wang, K. Zhang,
  \it  Infinitely many self-similar blow-up profiles for the Keller-Segel system in dimensions 3 to 9,
       \rm  arXiv:2503.02263.

\bibitem{quittner_souplet}
    \sc P. Quittner, Ph. Souplet,
    \rm Superlinear parabolic problems. Blow-up, global existence and steady states,
Second Edition. Birkh\"auser Advanced Texts, 2019.

  \bibitem{RS}
  \sc  P. Rapha\"el, R. Schweyer, 
  \it  On the stability of critical chemotactic aggregation,
  \rm Math. Ann. 359 (2014), 267-377.
  
  \bibitem{SSV}
  \sc  Y. Seki, Y. Sugiyama, J.J.L Vel\'azquez, 
  \it Multiple peak aggregations for the Keller-Segel system,
  \rm Nonlinearity 26 (2013), 319-352.
  
\bibitem{Sen05}
  \sc T. Senba, 
  \it Blow-up behavior of radial solutions to J\"ager-Luckhaus system in high dimensional domain,
      \rm Funkcial. Ekvac.  48 (2005), 247-271.
   
 \bibitem{Suzuki-2011}
  \sc T. Senba, T. Suzuki, 
    \it Applied analysis, second edition,
      \rm Imp. Coll. Press, London, 2011 World Sci. Publ., Hackensack, NJ, 2011

\bibitem{SouProc}
  \sc Ph. Souplet, 
    \it The influence of gradient perturbations on blow-up asymptotics in semilinear parabolic problems: a survey.
    \rm Progr. Nonlinear Differential Equations Appl. 64, Birkh\"auser, Basel, 2005.
  
  \bibitem{Soup-Win}
  \sc Ph. Souplet, M. Winkler,
\it Blow-up profiles for the parabolic-elliptic Keller-Segel system in dimensions $n\ge 3$,
  \rm Comm. Math. Phys. 367 (2019), 665-681.

\bibitem{SuzBook}
  \sc T. Suzuki, 
  \it Free energy and self-interacting particles.
  \rm Progr. Nonlinear Differential Equations Appl., 62, Birkh\"auser Boston Inc., Boston, MA, 2005.

\bibitem{Vel}
  \sc J.J.L. Vel\'azquez,
   \it  Stability of some mechanisms of chemotactic aggregation,
   \rm   SIAM J. Appl. Math. 62 (2002), 1581-1633.
   
  \bibitem{Win23}
  \sc M. Winkler,
     \it  Solutions to the Keller-Segel system with non-integrable behavior at spatial infinity,
   \rm  J. Elliptic Parabol. Equ. 9 (2023), 919-959.

\end{thebibliography}
\end{document}